\documentclass[11pt]{amsart}

\usepackage{amsmath,amsfonts,amssymb,eucal}

\newcommand{\ind}{\mathrm{Ind}}

\newcommand{\ennd}{\mathrm{End}}
\newcommand{\germ}{\mathfrak}
\newcommand{\spn}{\mathrm{Span}}

\newcommand{\ad}{\mathrm{ad}}
\newcommand{\cccs}{C^{\pi_i}(\mathbb R,\mathcal K_i)}
\newcommand{\cccsi}{C^{\pi}(\mathbb R,\mathcal K)}
\newcommand{\cccc}{C_c^{\pi_i}(\mathbb R,\mathcal K_i)}
\newcommand{\ccca}{C_c^{\pi}(\mathbb R,\mathcal K)}
\newcommand{\D}{\EuScript{D}}
\numberwithin{equation}{section}
\newtheorem{theorem}{\bf Theorem}[section]
\newtheorem{corollary}[theorem]{\bf Corollary}
\newtheorem{lemma}[theorem]{\bf Lemma}
\newtheorem{definition}[theorem]{\bf Definition}
\newtheorem{proposition}[theorem]{\bf Proposition}

\title{Unitary representations of nilpotent super Lie groups}
\author{Hadi Salmasian}
\date{Revised : 31 October, 2009}
\begin{document}
\maketitle
\begin{abstract}
We show that
irreducible unitary representations 
of nilpotent super Lie groups 
can be obtained by induction from
a distinguished class of sub super Lie groups. 
These sub super Lie groups 
are natural analogues of polarizing subgroups that appear in
classical Kirillov theory. 
We obtain a concrete geometric parametrization of 
irreducible unitary
representations by nonnegative definite coadjoint orbits.
As an application, we prove an analytic  
generalization of the 
Stone-von Neumann theorem for Heisenberg-Clifford super Lie groups. 
\end{abstract}
\section{Introduction}
\label{introd}
\subsection{Background}
One of the most elegant results in the theory of unitary representations
is the Stone-von Neumann theorem, which yields a classification of 
irreducible unitary representations of the Heisenberg group. It is 
the starting point in the study of unitary representations of 
nilpotent Lie groups, in which it
plays an essential role as well.

Kirillov's seminal work on unitary representations of nilpotent Lie
groups showed that unitary representations can be obtained in a simple fashion,
namely as induced representations from one-dimensional representations
of certain subgroups which are called \emph{polarizing subgroups}. 
From this, Kirillov deduced 
a well-behaved correspondence between irreducible unitary 
representations and coadjoint orbits.

Physicists are interested in unitary representations of Lie superalgebras and
super Lie groups
\footnote{
We follow \cite{delignemorgan} and \cite{vara} 
in using the terms \emph{super Lie group} and \emph{sub super Lie group}. 
Neverthelerss, instead of Deligne and Morgan's 
\emph{super Lie algebra} 
we use the term \emph{Lie superalgebra} merely because 
the latter is used in the literature more frequently.}
and their applications, e.g. in the classification of 
free relativistic super particles in SUSY quantum mechanics 
(see \cite{SS74} and \cite{FSZ81}). Extensions of the Stone-von Neumann theorem
to the Heisenberg-Clifford super Lie group, and the oscillator representation to the
orthosymplectic case, have been studied widely by physicists as well as
mathematicians (see \cite{nishiyama}, \cite{lo}, and \cite{Gunaydin}). 

Nevertheless,
much of the work done on infinite-dimensional unitary representations of 
super Lie groups treats representations algebraically, without addressing 
the analytic aspects. When the even part of the 
Lie superalgebra is a 
reductive Lie algebra
(e.g., for classical simple Lie superalgebras)
the space of the representation can be identified
with the space of 
$K$-finite analytic vectors of a unitary representation of 
the even part on a Hilbert 
space. This approach has been pursued in \cite{furutsunishiyama}.
However, this method is not applicable to more general super Lie groups,
e.g., the nilpotent ones.

Motivated by establishing a rigorous formalism for Mackey-Wigner's little group 
method in the super setting, the authors of \cite{vara} establish analytic
foundations of the theory of unitary representations of super Lie groups.
The key observation is that for infinite-dimensional representations, 
the action of the odd part of the Lie superalgebra 
is by \emph{unbounded} operators, and thus one
should consider densely defined operators.
As shown in \cite{vara}, it turns out that 
the correct space to realize the action
of the odd part is the (dense) subspace of \emph{smooth} vectors 
(in the sense of \cite[p. 52]{knapp}) for the even part.

\subsection{Our main results} The main goal of this work is 
to show that irreducible unitary representations 
(in the sense of \cite{vara}) of nilpotent
super Lie groups can be described in a way which is very similar
to the classical work of Kirillov. 
More specifically, our results are as follows. 
\begin{itemize}
\item[(a)] We generalize the notion of 
polarizing subalgebras of nilpotent Lie algebras 
to what we call \emph{polarizing systems} in 
nilpotent super Lie groups. 
Let $(N_0,\germ n)$ be a nilpotent super Lie group.
A polarizing system of $(N_0,\germ n)$ is a 6-tuple 
$$
(M_0,\germ m,\Phi,C_0,\germ c,\lambda)
$$
where $(M_0,\germ m)$ is a sub super Lie group of $(N_0,\germ n)$, 
$\Phi:(M_0,\germ m)\to(C_0,\germ c)$ is a homomorphism onto a 
super Lie group of Clifford type, and $\lambda\in\germ n_0^*$.
(There are extra conditions which are stated
in Definition \ref{polarizing}.)
We show that every irreducible unitary representation of a nilpotent super Lie group is induced 
from a special Clifford module associated
to a polarizing system (see part (a) of Theorem \ref{thmindpol}). 
The latter module is said to be 
\emph{consistent} with the
polarizing system.
Conversely, we prove that induction from a consistent representation of 
a polarizing system  
always results in an irreducible unitary 
representation (see Theorem \ref{lastmainthm}).
In other words, we show that induction yields the following
surjective map:
$$
\left\{
\begin{array}{c}
\textrm{Ordered pairs of polarizing}\\
\textrm{systems of }(N_0,\germ n)\textrm{ and their}\\
\textrm{consistent representations}
\end{array}
\right\}
\ \longrightarrow\ 
\left\{
\begin{array}{c}
\textrm{Irreducible unitary}\\
\textrm{representations of }(N_0,\germ n)\\
\textrm{up to unitary equivalence}
\end{array}
\right\}
$$
\item[(b)] 
Given a $\lambda\in\germ n_0^*,$
we obtain a simple necessary and sufficient condition for 
the existence of a polarizing system 
$(M_0,\germ m,\Phi,C_0,\germ c,\lambda)$ with 
a consistent representation. For every $\lambda\in\germ n_0^*$,
consider the symmetric bilinear form 
$$
\mathrm B_\lambda:\germ n_1\times\germ n_1\to\mathbb R
$$
defined by $\mathrm B_\lambda(X,Y)=\lambda([X,Y])$.
In Section \ref{polarisingsystems} we prove that such a 
polarizing system with a consistent representation exists 
if and only if 
$\mathrm B_\lambda$ is nonnegative definite.
\item[(c)]
We obtain a concrete geometric parametrization of 
irreducible unitary representations of $(N_0,\germ n)$.
Let 
$$
\germ n_0^{+}=\{\,\lambda\in\germ n_0^*\ |\ \mathrm B_\lambda
\textrm{ is nonnegative definite\,}\}.
$$
It is easily checked that 
$\germ n_0^{+}$ is a union of coadjoint orbits. 
In Theorem \ref{mthm} we prove that
the inducing construction outlined above yields 
the following
bijective correspondence:
$$
\Big\{N_0\textrm{-orbits in }\germ n_0^+
\Big\}
\longleftrightarrow
\left\{
\begin{array}{c}
\textrm{Irreducible unitary representations of}\\
(N_0,\germ n)\textrm{ up to unitary equivalence}\\
\textrm{and parity change}
\end{array}
\right\}
$$
\item[(d)] As a simple application of our results, 
we obtain a proof of an analytic formulation
of the Stone-von Neumann theorem for Heisenberg-Clifford super
Lie groups (see Section \ref{svngeneral}). 
We believe that
this analytic formulation is new.
We would like to mention that in \cite{rosenberg}, the author studies
a generalization of the Stone-von Neumann theorem to the super case.
Our approach has the advantage that it yields a 
concrete statement based on
a rigorous and more general notion of unitary representation for
super Lie groups, and avoids the 
assumption that the odd part have even dimension.
\item[(e)] A consequence of part (b) of Theorem \ref{thmindpol}
is a numerical invariant
of irreducible unitary representations of nilpotent super Lie groups.
The value of the invariant is a positive integer, and is equal to one 
if and only if
the representation is \emph{purely even}, i.e., in its 
$\mathbb Z_2$-grading the odd summand is trivial. 
\end{itemize}

In conclusion, this work is yet another 
justification for 
fruitfulness of the approach pursued in \cite{vara} to define and study unitary 
representations of super Lie groups rigorously.

\subsection{Organization of the paper}
This paper is organized as follows. Section \ref{notationandprem}
is devoted to recalling some basic definitions and facts about
super Lie groups and their unitary representations.
In Section \ref{specialindsec} we recall the notion  
of 
induction of unitary representations from \emph{special} sub super 
Lie groups which was introduced in \cite{vara}, and prove
that it can be done in stages (see Proposition \ref{indtransitivity}).
Section \ref{nilpsupersec} is devoted to studying the structure
of nilpotent super Lie groups, proving a version of Kirillov's
Lemma, and classification of representations
of super Lie groups of Clifford type. Section \ref{inductionresultfive}
contains a technical but important result. In this section we 
prove that under certain 
conditions a unitary representation
is induced from a sub super Lie group of codimension one. 
Although this result is analogous to one of Kirillov's original results,
there are several delicate issues involving unbounded operators which
need to be dealt with.
Using
the main result of Section \ref{inductionresultfive}, in Section
\ref{svngeneral} 
we obtain a proof of an analytic formulation of the 
Stone-von Neumann theorem for Heisenberg-Clifford super Lie groups.
In Section \ref{mainressecsix} we define polarizing systems,
prove the surjectivity of the map from induced representations 
to irreducible representations,
and show that if two polarizing systems yield the same representation 
then they correspond to the same coadjoint orbit. In Section \ref{suitable}
we prove the existence of a special kind of polarizing Lie subalgebra 
of $\germ n_0$. This section is fairly technical and contains 
several lemmas, but the main point is to prove Lemma 
\ref{existenceofevenpart}. In Section
\ref{polarisingsystems} we state
and prove our main result on parametrization of representations 
by coadjoint orbits (see Theorem \ref{mthm}).

\subsection{Acknowledgement}
After the first draft of this article was written,
we realized that M. Duflo had previously 
worked on the same problem and obtained similar 
results which were not published. 
We would like to thank M. Duflo for 
extremely illuminating conversations, 
and his encouragement 
to write this article.

\section{Preliminaries}
\label{notationandprem}
\subsection{Notation and basic definitions}
Recall that a densely defined operator $T$
on a Hilbert space $\mathcal H$ is called symmetric if 
for every $v,w\in D(T)$ we have 
$\langle Tv,w\rangle=\langle v,Tw\rangle$,
where $D(T)$ denotes the domain of $T$.

By a $\mathbb Z_2$-graded Hilbert space we mean a Hilbert space
$\mathcal H$ with an orthogonal decomposition 
$$
\mathcal H=\mathcal H_0\oplus\mathcal H_1.
$$
A densely defined linear
operator $T$ on $\mathcal H$ is called even 
(respectively, odd) 
if its domain $D(T)$ is $\mathbb Z_2$-graded, i.e., 
$$
D(T)=D(T)_0\oplus D(T)_1
$$ 
where for every $i\in\{0,1\}$ we have 
$D(T)_i=D(T)\cap \mathcal H_i$,
and for every $v\in D(T)_i$ we have
$Tv\in \mathcal H_i$ 
(respectively, $Tv\in\mathcal H_{1-i}$).

For basic definitions and facts about Lie superalgebras, we refer the 
reader to \cite{delignemorgan} and \cite{kac}. Unless explicitly 
stated otherwise, in this paper all 
Lie algebras and Lie superalgebras are over $\mathbb R$.

If $\germ g$ is a Lie superalgebra, its centre and universal enveloping algebra
are denoted by $\mathcal Z(\germ g)$ and $\mathcal U(\germ g)$.
Similarly, the centre of a Lie group $G$ is denoted by $\mathcal Z(G)$.
If a Lie group $G$ acts on a vector space $\EuScript{V}$, then the  
action of 
an element $g\in G$ on a vector $v\in\EuScript{V}$ is 
denoted by $g\cdot v$.

Following \cite{delignemorgan}, our definition of 
a super Lie group is based on
the notion of a \emph{Harish-Chandra pair}.
One can define a super Lie group concretely as follows.
\begin{definition}
\label{defnofliesupgp} 
A super Lie group is a pair $(G_0,\germ g)$ with the 
following properties.
\begin{itemize}
\item[(a)] $\germ g=\germ g_0\oplus \germ g_1$ is a 
Lie superalgebra over $\mathbb R$.
\item[(b)] $G_0$ is a connected 
real Lie group with Lie algebra $\germ g_0$ which
acts on $\germ g$ smoothly via $\mathbb R$-linear automorphisms. 
\item[(c)] The action of $G_0$ on $\germ g_0$ is the adjoint action.
The adjoint action of $\germ g_0$ on $\germ g$ is the 
differential of the action of $G_0$ on $\germ g$. 
\end{itemize}
\end{definition}

A super Lie group $(H_0,\germ h)$ is called a \emph{sub super Lie group} of 
a super Lie group $(G_0,\germ g)$ if $H_0$ is a Lie subgroup of 
$G_0$ and $\germ h=\germ h_0\oplus\germ h_1$ is a  
subalgebra
\footnote{
In this paper, instead of the term \emph{sub super Lie algebra}
of \cite{delignemorgan} we use the abbreviation \emph{subalgebra}.
}
of $\germ g$ such that $\germ h_0$
is the Lie subagebra of $\germ g_0$ corresponding to $H_0$  
and 
the action of $H_0$ on $\germ h$ is the
restriction of the action of $G_0$ on $\germ g$.

Let $(G_0,\germ g)$ and $(G_0',\germ g')$ be arbitrary super Lie groups.
A homomorphism 
$$
\Phi:(G_0,\germ g)\to(G_0',\germ g')
$$ 
consists of a Lie group 
homomorphism
from $G_0$ to $G_0'$ and a homomorphism of Lie superalgebras from $\germ g$ to $\germ g'$
which are compatible with each other.
We say $\Phi$ is surjective if both of these homomorphisms 
are surjective in the usual sense.

If $(\pi,\mathcal H)$ is a unitary representation of a Lie group $G$ on a Hilbert 
space $\mathcal H$, then the subspace of smooth vectors of $\mathcal H$ 
for the action of
$G$ is denoted by $\mathcal H^\infty$. The infinitesimal action of the Lie algebra of 
$G$ on $\mathcal H^\infty$ is denoted by $\pi^\infty$.

The definition of unitary representations of super Lie groups,
which is given below, is originally introduced in \cite{vara}. 
\begin{definition}
\label{unirep}
A unitary representation of $(G_0,\germ g)$ is a triple
$(\pi,\rho^\pi,\mathcal H)$ such that   
$\mathcal H$ 
is a $\mathbb Z_2$-graded Hilbert space 
endowed with 
a unitary representation $\pi$ of $G_0$, and  
$\rho^\pi:\germ g_1\to \ennd_\mathbb C(\mathcal H^\infty)$ 
is an $\mathbb R$-linear map
with the following properties.
\begin{itemize}
\item[(a)] For every $g\in G_0$, $\pi(g)$ is an even operator on $\mathcal H$.  
\item[(b)]
For every $X\in\germ g_1$, $\rho^\pi(X)$ is an odd linear operator. Moreover,
$\rho^\pi(X)$ is symmetric, i.e., for every $v,w\in\mathcal H^\infty$, we have
$$
\langle\rho^\pi(X)v,w\rangle=\langle v,\rho^\pi(X)w\rangle.
$$
\item[(c)] For every $X,Y\in\germ g_1$ and $v\in\mathcal H^\infty$, we have
$$
\rho^\pi(X)\rho^\pi(Y)v+\rho^\pi(Y)\rho^\pi(X)v=-\sqrt{-1}\pi^\infty([X,Y])v.
$$
\item[(d)] For every $g\in G_0$ and $X\in\germ g_1$, we have
$$\rho^\pi(g\cdot X)=\pi(g)\rho^\pi(X)\pi(g^{-1}).$$

\end{itemize}
 
\end{definition}
{\noindent\bf Remark.} 1. One can combine $\rho^\pi$ and $\pi^\infty$
to obtain
a representation of $\germ g$ in $\mathcal H^\infty$ where
an element $X_0+X_1\in\germ g_0\oplus\germ g_1$ acts by
$\pi^\infty(X_0)+e^{{\pi\over 4}\sqrt{-1}}\rho^\pi(X_1)$.
Consequently, from \cite[Proposition 1]{vara}
it follows that for every $X\in\germ g_0$, $Y\in\germ g_1$, and $v\in\mathcal H^\infty$
we have
$$
\rho^\pi([X,Y])v=\pi^\infty(X)\rho^\pi(Y)v-\rho^\pi(Y)\pi^\infty(X)v.
$$
2. From the closed graph theorem for  
Fr\'{e}chet spaces, it follows that for every 
$X\in\germ g_1$, $\rho^\pi(X)$ is a continuous operator on
$\mathcal H^\infty$.\\

Given two unitary representations $(\pi,\rho^\pi,\mathcal H)$ and 
$(\pi',\rho^{\pi'},\mathcal H')$ of $(G_0,\germ g)$, by
an intertwining operator from $(\pi,\rho^\pi,\mathcal H)$
to $(\pi',\rho^{\pi'},\mathcal H')$ we mean an even 
bounded linear transformation $T:\mathcal H\to\mathcal H'$ 
such that for every $g\in G_0$ and $X\in\germ g_1$
we have
$T\pi(g)=\pi'(g)T$ and $T\rho^\pi(X)=\rho^{\pi'}(X) T$.
(Note that if $T\pi(g)=\pi'(g)T$ for every $g\in G_0$, then 
$T\mathcal H^\infty\subseteq\mathcal H'^\infty$.)

Two unitary representations $(\pi,\rho^\pi,\mathcal H)$ and 
$(\pi',\rho^{\pi'},\mathcal H')$ of $(G_0,\germ g)$ are said to be unitarily 
equivalent if there exists an isometry 
$T:\mathcal H\to \mathcal H'$
which is also an intertwining operator.
Note that it follows that $T\mathcal H^\infty= \mathcal H'^\infty$.

From now on, to indicate that two unitary representations 
$(\pi,\rho^\pi,\mathcal H)$ and 
$(\pi',\rho^{\pi'},\mathcal H')$ are unitarily equivalent, we write 
$$
(\pi,\rho^\pi,\mathcal H)\simeq (\pi',\rho^{\pi'},\mathcal H').
$$

A unitary representation $(\pi,\rho^\pi,\mathcal H)$ of $(G_0,\germ g)$
is called irreducible if $\mathcal H$ does not have any proper 
$(G_0,\germ g)$-invariant
closed $\mathbb Z_2$-graded subspaces.
By \cite[Lemma 5]{vara}, a representation $(\pi,\rho^\pi,\mathcal H)$
is irreducible if and only if every intertwining operator
from 
$(\pi,\rho^\pi,\mathcal H)$ to itself is scalar.

From every 
unitary representation $(\pi,\rho^\pi,\mathcal H)$ we can obtain a new
unitary representation $(\pi,\rho^\pi,\mathrm{\Pi}\mathcal H)$
where $\mathrm{\Pi}$ is the parity change operator. 
The operator $\mathrm{\Pi}$
can be considered as
a special case of tensor product, namely tensoring 
$(\pi,\rho^\pi,\mathcal H)$
with a trivial ($0|1$)-dimensional representation.
The unitary 
representations $(\pi,\rho^\pi,\mathcal H)$
and $(\pi,\rho^\pi,\mathrm{\Pi}\mathcal H)$ are said to be 
the same up to parity change.
Note that they are not necessarily 
unitarily equivalent. 

From now on, to indicate that two unitary representations $(\pi,\rho^\pi,\mathcal H)$ and 
$(\pi',\rho^{\pi'},\mathcal H')$ are identical up to unitary 
equivalence and parity change,  
we write 
$$
(\pi,\rho^\pi,\mathcal H)\eqsim(\pi',\rho^{\pi'},\mathcal H').
$$

\subsection{Stability of unitary representations}

A remarkable feature of unitary representations 
as defined in Definition \ref{unirep} is their stability, i.e. that
one can replace the space $\mathcal H^\infty$ with a variety of
dense and invariant 
subspaces. Stability is needed even for justifying 
that the restriction of a unitary representation of a 
super Lie group $(G_0,\germ g)$ to a sub super 
Lie group $(H_0,\germ h)$ is well defined, i.e., 
that the restriction determines a unique unitary 
representation up to unitary equivalence.
The result that justifies the latter statement 
is \cite[Proposition 2]{vara}. 
For the reader's convenience, and for 
future reference in this article, we would
like to record a slightly simplified formulation 
of the statement of this proposition.
\begin{proposition}
\label{varadarajan}
{\rm\cite[Proposition 2]{vara}}
Let $(G_0,\germ g)$ be a super Lie group and
$(\pi,\mathcal H)$ be a unitary representation of $G_0$.
Suppose $\mathcal B$ is a dense, $\mathbb Z_2$-graded, and 
$G_0$-invariant subspace of $\mathcal H$, and
$\{\rho(X)\}_{X\in\germ g_1}$ is a family
of densely defined linear operators on $\mathcal H$ with the
following properties.
\begin{itemize}
\item[(a)] $\mathcal B\subseteq \mathcal H^\infty$. 
\item[(b)] If $X\in\germ g_1$, then
$\mathcal B\subseteq D(\rho(X))$. 
\item[(c)] $\rho(X)$ is symmetric for every $X\in\germ g_1$.
\item[(d)] For every $X\in\germ g_1$ and $i\in\{0,1\}$ we have 
$\rho(X)\mathcal B_i\subseteq \mathcal H_{1-i}$.
\item[(e)] If $X,Y\in\germ g_1$ and $a,b\in\mathbb R$ then
$\rho(aX+bY)=a\rho(X)+b\rho(Y)$.
\item[(f)] $\pi(g)\rho(X)\pi(g^{-1})=\rho(g\cdot X)$ for all
$g\in G_0$ and $X\in\germ g_1$.
\item[(g)]  For every $X,Y\in \germ g_1$ we have 
$\rho(X)\mathcal B\subseteq D(\rho(Y))$.
\item[(h)] For every $X,Y\in \germ g_1$ and $v\in\mathcal B$ we have
$$
\rho(X)\rho(Y)v+\rho(Y)\rho(X)v=-\sqrt{-1}\pi^\infty([X,Y])v.
$$
\end{itemize} 
Then the following statements hold.
\begin{itemize}
\item[(i)]
For every $X\in\germ g_1$, the operator $\rho(X)$ is essentially
self adjoint, and the closure $\overline{\rho(X)}$  of $\rho(X)$ satisfies
$\mathcal H^\infty\subseteq D(\overline{\rho(X)})$. 
\item[(ii)]
Suppose that for every $X\in\germ g_1$
and $v\in\mathcal H^\infty$, we set
$\rho^\pi(X)v=\overline{\rho(X)}v$.
Then for every $X\in\germ g_1$ we have $\rho^\pi(X)\in\ennd_\mathbb C(\mathcal H^\infty)$. 
Moreover, $(\pi,\rho^\pi,\mathcal H)$ is a unitary representation of $(G_0,\germ g)$.
\item[(iii)]
Let
$(\pi',\rho^{\pi'},\mathcal H)$ be a unitary representation of
$(G_0,\germ g)$ in the same 
$\mathbb Z_2$-graded 
Hilbert space $\mathcal H$. 
Suppose that for every $g\in G_0$ we have $\pi'(g)=\pi(g)$,
and for every 
$X\in\germ g_1$ and $v\in\mathcal B$ we have
$\rho^{\pi'}(X)v=\rho^\pi(X)v$.
Then $(\pi',\rho^{\pi'},\mathcal H)\simeq(\pi,\rho^\pi,\mathcal H)$,
and the intertwining isometry $T:\mathcal H\to\mathcal H$ 
yielding this unitary equivalence is the identity map.
\end{itemize}
\end{proposition}


We conclude this section with a simple but useful lemma
about polarizing subalgebras. Let $\germ g$ be a Lie algebra and
fix $\lambda\in\germ g^*$. Recall that a 
Lie subalgebra $\germ m$ of $\germ g$ is called
a polarizing subalgebra corresponding to $\lambda$ if 
$\germ m$ is a maximal isotropic subspace of $\germ g$ 
for the skew-symmetric 
bilinear form $\omega_\lambda:\germ g\times\germ g\to\mathbb R$
defined by $\omega_\lambda(X,Y)=\lambda([X,Y])$.
\begin{lemma}
\label{nonzeropolar}
Let $\germ g$ be a nilpotent Lie algebra, $\lambda\in\germ g^*$, and 
$\germ m\subseteq \germ g$ be a Lie subalgebra. If $\lambda\neq 0$ and
$\germ m$ is a polarizing subalgebra of $\germ g$ corresponding to $\lambda$,
then there exists an $X\in\germ m$ such that $\lambda(X)\neq 0$.
\end{lemma}
\begin{proof}
Suppose, on the contrary, that $\germ m\subseteq\mathrm{ker}\,\lambda$.
Then from 
$\lambda\neq 0$ it follows that $\germ g\supsetneq\germ m$. If 
$$
N_\germ g(\germ m)=\{Y\in\germ g\ |\ [Y,\germ m]\subseteq \germ m\}
$$
then $N_\germ g(\germ m)$ is a Lie subalgebra of $\germ g$ and 
$N_\germ g(\germ m)\supsetneq\germ m$. Choose an $X\in N_\germ g(\germ m)$
such that $X\notin \germ m$ and set $\germ m'=\germ m\oplus\mathbb RX$. 
It is easy to check that $[\germ m',\germ m']\subseteq \germ m$
and thus $\lambda([\germ m',\germ m'])=\{0\}$ which
contradicts maximality of dimension of $\germ m$.

\end{proof}

\section{Special induction}
\label{specialindsec}
\subsection{Realization of the induced representation}
\label{realization}
Let $(G_0,\germ g)$ be a super Lie group and $(H_0,\germ h)$ be 
a sub super Lie group of $(G_0,\germ g_0)$, i.e., $H_0\subseteq G_0$
and $\germ h\subseteq \germ g$. As in \cite[\S3.2]{vara}, we assume that 
$(H_0,\germ h)$ is \emph{special}, i.e., that 
$\germ h_1=\germ g_1$. 
For every unitary representation 
$(\sigma,\rho^\sigma,\mathcal K)$ of $(H_0,\germ h)$, 
the representation  
of $(G_0,\germ g)$ induced from $(\sigma,\rho^\sigma,\mathcal K)$
is defined 
in \cite[\S3]{vara}. We recall the definition of special
induction only in a case which we need in this paper, i.e., when 
the Lie groups
$G_0$ and $H_0$ are unimodular.
In this case, to define the representation 
$(\pi,\rho^\pi,\mathcal H)$ of $(G_0,\germ g)$  
induced from 
$(\sigma,\rho^\sigma,\mathcal K)$,
one fixes a $G_0$-invariant 
measure $\mu$ on $H_0\backslash G_0$ and defines $\mathcal H$
as the space of measurable functions $f:G_0\to \mathcal K$ such that
\begin{itemize}
\item[(a)] For any $g\in G_0$ and $h\in H_0$, we have 
$f(hg)=\sigma(h)f(g)$.
\item[(b)] $\int_{H_0\backslash G_0}||f(g)||^2d\mu<\infty$.
\end{itemize}
The action of every $g\in G_0$ on every $f\in\mathcal H$ is the usual 
right regular representation, i.e.,
$$
\textrm{if }g,g'\in G_0\textrm{ then }\big(\pi(g)f\big)\,(g')=f(g'g).
$$ 
Recall that $\mathcal H^\infty$ is 
the space of smooth vectors of $(\pi,\mathcal H)$.
It is well-known that 
$$
\mathcal H^\infty\subseteq C^\infty(G_0,\mathcal K),
$$ 
where
$C^\infty(G_0,\mathcal K)$ denotes the space of smooth functions 
$f:G_0\to \mathcal K$
(see \cite[Theorem 5.1]{poulsen} or 
\cite[Theorem A.1.4]{corgr}). Moreover, one can check that 
for every $f\in\mathcal H^\infty$, we have $f(G_0)\subseteq \mathcal K^\infty$.

Let $\mathcal H^{\infty,c}$ be the space consisting of 
functions $f:G_0\to \mathcal K$ such that  
$$
f\in\mathcal H\cap C^\infty(G_0,\mathcal K)
$$ 
and $\mathrm{Supp}(f)$ is compact modulo $H_0$.
It is shown in 
\cite[Proposition 4]{vara} that $\mathcal H^{\infty,c}\subseteq \mathcal H^\infty$,
the subspace $\mathcal H^{\infty,c}$ is dense in 
$\mathcal H$, and for every $X\in\germ g_0$ we have
$
\pi^\infty(X)\mathcal H^{\infty,c}\subseteq\mathcal H^{\infty,c}
$.
The action of $\germ g_1$ is initially defined on $\mathcal H^{\infty,c}$.
For every $X\in\germ g_1$ and 
$f\in\mathcal H^{\infty,c}$, one defines 
\begin{equation}
\label{oddpartaction}
\big(\rho^\pi(X)f\big)\,(g)=\rho^\sigma(g\cdot X)\big(f(g)\big).
\end{equation}
From Proposition \ref{varadarajan}
it follows that the domain of the closure of 
$\rho^\pi(X)$ contains $\mathcal H^\infty$, and consequently
the induced representation $(\pi,\rho^\pi,\mathcal H)$ is 
well-defined. We will denote the latter representation by
$$
\ind_{(H_0,\germ h)}^{(G_0,\germ g)}(\sigma,\rho^\sigma,\mathcal K).
$$

\subsection{Special induction in stages}
A basic but important property of special induction is that
it can be done in stages. The proof of this property is not difficult, but
it is not mentioned in \cite{vara} explicitly. For the reader's convenience,
we would like to sketch it below.
\begin{proposition}
\label{indtransitivity}
Suppose that $(G_0,\germ g)$ is a super Lie group,
$(K_0,\germ k)$ is a special sub super Lie group of $(G_0,\germ g)$,
and $(H_0,\germ h)$ is a special sub super Lie group of $(H_0,\germ h)$.
Assume that $G_0,K_0$, and $H_0$ are unimodular, and let 
$(\sigma,\rho^\sigma,\mathcal K)$ be a unitary repsentation of 
$(H_0,\germ h)$. Then
\begin{equation}
\label{transi}
\ind_{(H_0,\germ h)}^{(G_0,\germ g)}(\sigma,\rho^\sigma,\mathcal K)
\simeq\ind_{(K_0,\germ k)}^{(G_0,\germ g)} 
\ind_{(H_0,\germ h)}^{(K_0,\germ k)}(\sigma,\rho^\sigma,\mathcal K).
\end{equation}
\end{proposition}
\begin{proof} Set
$$
(\pi,\rho^\pi,\mathcal H)=\ind_{(H_0,\germ h)}^{(G_0,\germ g)}(\sigma,\rho^\sigma,\mathcal K)\,,\ \    
(\eta,\rho^\eta,\mathcal L)=\ind_{(H_0,\germ h)}^{(K_0,\germ k)}(\sigma,\rho^\sigma,\mathcal K),$$ 
and 
$$
(\nu,\rho^\nu,\mathcal V)=\ind_{(K_0,\germ k)}^{(G_0,\germ g)}
(\eta,\rho^\eta,\mathcal L).
$$
Thus $\mathcal H, \mathcal L$, and $\mathcal V$ are function
spaces introduced in Section \ref{realization} which realize the 
corresponding induced representations.
Let $\mathcal H^{\infty,c}$ be the subspace
of $\mathcal H$ defined in Section \ref{realization}.
We define $\mathcal L^{\infty,c}$ and $\mathcal V^{\infty,c}$ similarly.
The intertwining map
$$
T:(\pi,\mathcal H)\to(\nu,\mathcal V)
$$ 
is given in \cite[\S 4]{mackey}.
We recall the definition of $T$.
Given a function $f:G_0\to\mathcal K$ such that 
$f\in\mathcal H^{\infty,c}$,
the function 
$Tf:G_0\to \mathcal L$ 
is obtained as follows. For every $g\in G_0$ and $k\in K_0$,
$$
\big(Tf(g)\big)(k)=f(kg).
$$
One can normalize the involved measures such that for every 
$f\in\mathcal H^{\infty,c}$,
we have $||Tf||=||f||$. 

Fix an $f\in\mathcal H^{\infty,c}$ and a $V\in\germ g_1$. 
Since $f$ is a smooth vector for $\pi=\ind_{H_0}^{G_0}\sigma$
and $T$ is an interwining isometry, 
$Tf$ is a smooth vector
for 
$$
\nu=\ind_{H_0}^{G_0}\ind_{K_0}^{H_0}\sigma.
$$ 
By Proposition \ref{varadarajan}, 
to prove Proposition \ref{indtransitivity} it suffices to show that
\begin{equation}
\label{equationforactionss}
\textrm{for every }f\in\mathcal H^{\infty,c}\textrm{ and }
V\in\germ g_1,\ \ T\rho^\pi(V)f=\rho^\nu(V)Tf.
\end{equation}
Since $\mathrm{Supp}(f)$
is compact modulo $H_0$, it follows readily that $\mathrm{Supp}(Tf)$ is compact modulo
$K_0$, and for every $g\in G_0$ the support of 
$$
Tf(g):K_0\to\mathcal K
$$ is compact
modulo $H_0$. From \cite[Proposition 4]{vara}
it follows that $Tf\in\mathcal V^{\infty,c}$
and for every $g\in G_0$ we have $Tf(g)\in \mathcal L^{\infty,c}$.
By the definition of 
special induction given in Section \ref{realization}, 
the action of ${\rho^\nu}(V)$ on $Tf$ is 
given by (\ref{oddpartaction}).
For the same reason, the action of $\rho^\eta(V)$
on $Tf(g)$ is given by (\ref{oddpartaction}).
Thus for every $g\in G_0$, and $k\in K_0$,
\begin{eqnarray*}
\Big(\big(\rho^\nu(V)Tf\big)\,(g)\Big)(k)&=&
\Big(\,\rho^\eta(g\cdot V)\big(Tf(g)\big)\,\Big)(k)\\
&=&\rho^\sigma(kg\cdot V)
\Big(\,\big(Tf(g)\big)(k)\,\Big)=\rho^\sigma(kg\cdot V)\big(f(kg)\big).
\end{eqnarray*}
To finish the proof of (\ref{equationforactionss})  
note
that 
\begin{eqnarray*}
\Big(\,\big(T\rho^\pi(V)f\big)\,(g)\,\Big)(k)&=&\big(\rho^\pi(V)f\big)\,(kg)\\
&=&\rho^\sigma(kg\cdot V)\big(f(kg)\big)
=\Big(\,\big(\rho^\nu(V)Tf\big)\,(g)\,\Big)(k).
\end{eqnarray*}

\end{proof}

\section{Reduced forms and super Lie groups of Clifford type}
\label{nilpsupersec}

\subsection{The reduced form of a super Lie group}
\label{thereduced}

Contrary to the case of locally compact groups, 
super Lie groups do not necessarily have faithful representations.
The next lemma presents a simple but key example of such elements.
\begin{lemma}
Let $(G_0,\germ g)$ be a super Lie group and 
$(\pi,\rho^\pi,\mathcal H)$ be a unitary representation of
$(G_0,\germ g)$. If
$X_1,...,X_m\in\germ g_1$ satisfy
$$
\sum_{i=1}^m[X_i,X_i]=0
$$
then $\rho^\pi(X_i)=0$ for every $1\leq i\leq m$.
\end{lemma}
\begin{proof}
We have 
$$
\sum_{i=1}^m \rho^\pi(X_i)^2=-{\sqrt{-1}\over 2}
\sum_{i=1}^m\pi^\infty([X_i,X_i])=
-{\sqrt{-1}\over 2}\pi^\infty(\sum_{i=1}^m[X_i,X_i])=0.
$$ 
Since every $\rho^\pi(X_i)$ is symmetric, for every 
$v\in \mathcal H^\infty$ we have 
$$
\sum_{i=1}^m\langle \rho^\pi(X_i)v,\rho^\pi(X_i)v\rangle=
\langle v,\sum_{i=1}^m\rho^\pi(X_i)^2v\rangle=0.
$$
Therefore for every $1\leq i\leq m$ we have
$\langle \rho^\pi(X_i)v,\rho^\pi(X_i)v\rangle=0$, which
implies that $\rho^\pi(X_i)v=0$.

\end{proof}

The proof of the following proposition is easy by induction.
\begin{proposition}
\label{nilrad}
Let $(G_0,\germ g)$ be a super Lie group.
Let $\germ a^{(1)}$ be the ideal of $\germ g$ generated by 
all $X\in\germ g_1$ such that $[X,X]=0$. For every $m>1$,
let $\germ a^{(m)}$ be the ideal of $\germ g$ generated 
by elements $X\in\germ g_1$ such that $[X,X]\in\germ a^{(m-1)}$.
Then for every unitary representation $(\pi,\rho^\pi,\mathcal H)$
of $(G_0,\germ g)$, the $\mathbb Z_2$-graded ideal 
$\bigcup_{m=1}^\infty\germ a^{(m)}$
acts trivially on $\mathcal H$, i.e., 
$\rho^\pi(X)=0$ if $X\in\germ g_1\cap\bigcup_{m=1}^\infty\germ a^{(m)}$, 
and $\pi^\infty(X)=0$ 
if $X\in\germ g_0\cap \bigcup_{m=1}^\infty\germ a^{(m)}$. 
\end{proposition}
Note that $\germ a^{(1)}\subseteq\germ a^{(2)}\subseteq\cdots$
and the set 
$\germ a[\germ g]=\bigcup_{m=1}^\infty\germ a^{(m)}$ is 
a $\mathbb Z_2$-graded ideal
of $\germ g$.
Therefore we have $\germ a[\germ g]=\germ a[\germ g]_0
\oplus\germ a[\germ g]_1$. Let $A_0$
be the normal subgroup of $G_0$ with Lie algebra 
$\germ a[\germ g]_0$. Then 
$(A_0,\germ a[\germ g])$ is a sub super Lie group 
of $(G_0,\germ g)$. Moreover, 
setting $\overline G_0=G_0/A_0$
and $\overline{\germ g}=\germ g/\germ a[\germ g]$ we obtain a nilpotent
super Lie group $(\overline G_0,\overline{\germ g})$. 

From now on, the super Lie
group $(\overline G_0,\overline{\germ g})$ will be called 
the \emph{reduced form} of $(G_0,\germ g)$. Obviously, the categories of 
unitary representations of $(G_0,\germ g)$ and 
$(\overline G_0,\overline{\germ g})$ 
are equivalent.

If $(G_0,\germ g)$ is a super Lie group with the property that
$\germ a[\germ g]=\{0\}$, then the super Lie group
$(G_0,\germ g)$ and the Lie superalgebra $\germ g$ are called 
\emph{reduced}.

\subsection{Heisenberg-Clifford super Lie groups}
\label{exampcliff}
In this section we introduce an important example of nilpotent 
super Lie groups which will be used in the rest of this paper.

Let $(\germ w,\omega)$ be a super symplectic vector
space, i.e., a $\mathbb Z_2$-graded vector space 
$\germ w=\germ w_0\oplus\germ w_1$ endowed with a 
nondegenerate bilinear form 
$$
\omega:\germ w\times\germ w\to\mathbb R
$$
with the following properties.
\begin{itemize}
\item[(a)] $\omega(\germ w_0,\germ w_1)=\omega(\germ w_1,\germ w_0)=\{0\}$. 
\item[(b)] Restriction of $\omega$ to $\germ w_0$ is a symplectic
form.
\item[(c)] Restriction of $\omega$ to $\germ w_1$ is a symmetric
form.
\end{itemize}
Consider the $\mathbb Z_2$-graded vector space
$$
\germ n=\germ w\oplus\mathbb R
$$
where $\germ n_0=\germ w_0\oplus\mathbb R$ and $\germ n_1=\germ w_1$.
We define a (super)bracket on $\germ n$ as follows. 
For every $P,Q\in\germ w$ and $a,b\in\mathbb R$,
we set
$$
[\,(P,a),(Q,b)\,]=(\,0\,,\,\omega(P,Q)\,).
$$
One can easily check that with this bracket 
$\germ n$ becomes a Lie superalgebra.
The latter Lie superalgebra is called a Heisenberg-Clifford
Lie superalgebra. 
If $N_0$ denotes the simply connected nilpotent Lie group
with Lie algebra $\germ n_0$, then the super Lie group 
$(N_0,\germ n)$ is
called a Heisenberg-Clifford super Lie group.

It may sometimes be more convenient to work with an explicit 
basis of the Heisenberg-Clifford Lie superalgebra.
One can always find a basis
\begin{equation}
\label{basisiss}\{Z,X_1,\ldots,X_m,
Y_1,\dots,Y_m,V_1,\ldots,V_n\}
\end{equation}
of $\germ n$ such that 
\begin{itemize}
\item[(a)]
$\germ n_0=\spn_\mathbb R\{Z,X_1,\ldots,X_m,Y_1,\ldots,Y_m\}
\textrm{ and }
\germ n_1=\spn_\mathbb R\{V_1,\ldots,V_n\}$. 
\item[(b)] For every $1\leq i\leq m$ we have 
$[X_i,Y_i]=Z$.
\item[(c)] For every $1\leq j\leq n$ we have 
\begin{equation}
\label{exampcliffeqn}
[V_j,V_j]=c_jZ
\end{equation} 
where $c_j\in\{1,-1\}$.
\item[(d)] $Z\in\mathcal Z(\germ n)$.
\end{itemize}

\subsection{Nilpotent supergroups}
Recall that a Lie superalgebra $\germ g$ is called
nilpotent if $\germ g$ appears in its own 
upper central series. (Equivalently, $\germ g$ is called
nilpotent if its lower central series 
has only finitely many nonzero terms.)

In this paper, a super Lie group $(N_0,\germ n)$ is called nilpotent if
it has the following properties.
\begin{itemize}
\item[(a)] $\germ n$ is a
nilpotent Lie superalgebra. 
\item[(b)] $N_0$ is a connected, simply connected, nilpotent
Lie group.
\end{itemize}
It follows that the exponentional map $\exp:\germ n_0\to N_0$ is an 
analytic diffeomorphism
which results in a bijective correspondence between Lie subgroups 
and Lie subalgebras.

\subsection{Structure of reduced forms} Our next task is
to state and prove a generalization of Kirillov's lemma 
\cite[Lemma 1.1.12]{corgr}. The proof of this generalization is
a slight modification of that of the original result.
Recall that $\mathcal Z(\germ n)$ denotes the centre of $\germ n$.

\begin{definition}
\label{cliftypedef}
A nilpotent Lie superalgebra $\germ n$ is said to be of Clifford type
if one of the following properties hold.
\begin{itemize}
\item[(a)] $\germ n=\{0\}$.
\item[(b)] $\germ n$ is a Heisenberg-Clifford Lie superalgebra such that
$\dim\germ n_0=1$ and the restriction of $\omega$ to $\germ n_1$
is positive definite. 
\end{itemize}
\end{definition}
In other words, $\germ n$ is of Clifford type if
either $\germ n=\{0\}$ or 
$\germ n$ satisfies both of the following properties.
\begin{itemize}
\item[(a)] $\dim \germ n_0=1$ and $\mathcal Z(\germ n)=\germ n_0$. 
\item[(b)] There exists a basis 
\begin{equation}
\label{cliftypedeff}
\{Z,V_1,\ldots,V_l\}
\end{equation}
of $\germ n$ such that 
$Z\in\mathcal Z(\germ n_0)$, $V_1,\ldots,V_l\in\germ n_1$, and 
for every 
$1\leq i\leq j\leq l$ we have $[V_i,V_j]=\delta_{i,j}Z$.
\end{itemize}
A nilpotent super Lie group $(N_0,\germ n)$ is said to be of Clifford type
whenever $\germ n$ is of Clifford type.\\

Note that the zero-dimensional Lie superalgebra
and the (unique)
Lie superalgebra 
$\germ n$ which satisfies $\dim\germ n=\dim\germ n_0=1$ are
also considered to be of Clifford type. Up to parity change,
irreducible 
unitary representations of their 
corresponding super Lie groups are one-dimensional and purely even. 
Up to parity change, trivial representation is the only such
representation 
of the first case. For the second case, 
these representations are 
unitary characters of the even part.

\begin{proposition}
\label{kirillovlemma}
Let $\germ n$ be a reduced nilpotent Lie superalgebra which satisfies
$\dim\germ n>1$ and $\dim\mathcal Z(\germ n)=1$. Then exactly one of the 
following two statements is true. 
\begin{itemize}
\item[(a)]
There exist three nonzero elements $X,Y,Z\in\germ n_0$ such that
$$
\germ n=\mathbb RX\oplus\mathbb RY\oplus
\mathbb RZ\oplus \germ w
$$ 
where $\germ w=\germ w_0\oplus\germ w_1$ is a $\mathbb Z_2$-graded 
subspace of $\germ n$,  
$[X,Y]=Z$, and $Z\in\mathcal Z(\germ n)$. Moreover,
the vector space
$$
\germ n'=\mathbb RY\oplus\mathbb RZ\oplus\germ w
$$ is a 
subalgebra of $\germ n$, and $Y\in\mathcal Z(\germ n')$. 
\item[(b)] $\germ n$ is a Lie superalgebra of Clifford type.
\end{itemize}
\end{proposition}
\begin{proof}
Obviously $\mathcal Z(\germ n)\subseteq\germ n_0$, since if 
$X\neq 0$ and $X\in\germ n_1\cap\mathcal Z(\germ n)$, then $[X,X]=0$ which
contradicts the assumption that $\germ n$ is reduced. 

Fix an arbitrary nonzero $Z\in\mathcal Z(\germ n)$.
Since $\germ n$ is nilpotent,
we have 
$$
\mathcal Z(\germ n/\mathcal Z(\germ n))\neq\{0\}.
$$
Let $\mathcal Z^1(\germ n)$ denote the $\mathbb Z_2$-graded ideal of $\germ n$ 
which corresponds to $\mathcal Z(\germ n/\mathcal Z(\germ n))$
via the quotient map $\mathsf q:\germ n\to \germ n/\mathcal Z(\germ n)$.
Obviously $\mathcal Z^1(\germ n)\supsetneq\mathcal Z(\germ n)$.

First assume that 
$\mathcal Z^1(\germ n)\cap \germ n_0\nsubseteq \mathcal Z(\germ n)$.
We show that statement (a) of the proposition holds. 
Choose an arbitrary $Y\in\mathcal Z^1(\germ n)\cap \germ n_0$
such that $Y\notin\mathcal Z(\germ n)$, 
and consider the map $\ad_Y:\germ n\to \mathcal Z(\germ n)$. Since
$Y\in\germ n_0$ and $Y\notin\mathcal Z(\germ n)$, 
there exists an element $X\in \germ n_0$ such that $[X,Y]\neq 0$.
After an appropriate rescaling, we can assume that $[X,Y]=Z$. We can now take 
$\germ w$ to be a $\mathbb Z_2$-graded complement of $\mathbb RY\oplus\mathbb RZ$
in $\germ n'$, where
$$
\germ n'=\{\ V\in\germ n\ |\ [V,Y]=0\ \}.
$$

Next assume that 
$\mathcal Z^1(\germ n)\cap \germ n_0\subseteq\mathcal Z(\germ n)$. We
use induction on $\dim\germ n$ to show that $\germ n$  
is of Clifford type. 
Choose an arbitrary nonzero $V_1\in\mathcal Z^1(\germ n)\cap\germ n_1$. 
Since $\germ n$ is reduced, after an appropriate 
rescaling we can assume that $[V_1,V_1]=\pm Z$. 
If $\dim \germ n=2$, then
the proof is complete. Next assume $\dim \germ n>2$.
Set 
$$
\germ n'=\ker(\ad_{V_1}).
$$ Obviously $\germ n'$  
is a subalgebra of $\germ n$ and 
$\germ n=\germ n'\oplus\mathbb RV_1$ as vector spaces. 
It is easy to see that $\germ n'$ is reduced, 
$\mathcal Z(\germ n')=\mathcal Z(\germ n)$, and 
$\mathcal Z^1(\germ n')=\mathcal Z^1(\germ n)\cap \germ n'$. 
Therefore 
$\dim\mathcal Z(\germ n')=1$
and $\mathcal Z^1(\germ n')\cap\germ n'_0\subseteq\mathcal Z(\germ n')$. 
By induction hypothesis, there exists a basis 
$\{Z',V_2,\ldots,V_l\}$ for $\germ n'$ such that
$Z'\in\mathcal Z(\germ n')$, 
$V_l\in\germ n'_1$ for every $l>1$, 
and $[V_i,V_j]=\delta_{i,j}Z'$ 
for every $1<i\leq j\leq l$.
After rescaling $V_1$ appropriately, we have
$[V_1,V_1]=\pm Z'$. If $[V_1,V_1]=Z'$, then the proof is
complete. If $[V_1,V_1]=-Z'$, then it follows that 
$[V_1+V_2,V_1+V_2]=0$, contradicting the assumption that
$\germ n$ is reduced. 

\end{proof}
\subsection{Representations of super Lie groups of Clifford 
type}
\label{unitclifford}
Throughout this section, we assume 
that $(C_0,\germ c)$ is a super Lie group of Clifford type
such that $\germ c\neq\{0\}$. Let
$\{Z,V_1,\ldots,V_l\}$ be the basis of $\germ c$
given in (\ref{cliftypedeff}), and 
$(\pi,\rho^\pi,\mathcal H)$ be an irreducible unitary 
representation
of $(C_0,\germ c)$. By \cite[Lemma 5]{vara}, the action of $\rho^\pi(Z)$
is via multiplication by a scalar. It follows that $\mathcal H^\infty=\mathcal H$,
i.e., for every $1\leq i\leq l$ we have 
\begin{equation}
\label{fuldomain}
D(\rho^\pi(V_i))=\mathcal H.
\end{equation}
Fix $1\leq i\leq l$. Since $\rho^\pi(V_i)$ 
is symmetric, it is closable \cite[p. 316]{conway}. Thus 
(\ref{fuldomain}) implies
that $\rho^\pi(V_i)$ is a closed operator. Consequently, by the closed graph theorem, 
$\rho^\pi(V_i)$ is a bounded, self adjoint operator.

Since $\rho^\pi(V_1)$ is a self adjoint operator and 
$\pi^\infty(Z)=2\sqrt{-1}\rho^\pi(V_1)^2$, it follows that for every 
$v\in\mathcal H$ we have 
$$
\pi^\infty(Z)v=a\sqrt{-1}v
$$
where $a\geq 0$. 

If $a=0$, then for every $V\in\germ c_1$ the symmetric
operator 
$\rho^\pi(V)$ satisfies $\rho^\pi(V)^2=0$, and 
it follows immediately that $\rho^\pi(V)=0$.
Therefore $(\pi,\rho^\pi,\mathcal H)$ is a 
trivial representation. 

Next suppose $a>0$. Let $\langle Z-a\rangle$ denote the 
two-sided ideal of $\mathcal U(\germ c)$ 
generated by $Z-a$.
We can set $\rho(V)=\rho^\pi(V)$
for every $V\in\germ c_1$, and then extend $\rho$ to 
a homomorphism 
$$
\rho:\mathcal U(\germ c)/\langle Z-a\rangle\to\mathrm{End}_\mathbb C(\mathcal H)
$$
of associative (super)algebras. In this fashion, from a 
representation of $(C_0,\germ c)$ we obtain a representation of 
$\mathcal U(\germ c)/\langle Z-a\rangle$ on a complex $\mathbb Z_2$-graded 
vector space.

Fix a nonzero vector $v\in\mathcal H_0$ and consider the subspace
$\mathcal W\subseteq \mathcal H$ defined by
$$
\mathcal W=\spn_\mathbb C\{\rho(W)v\,|\,W\in\mathcal U(\germ c)\}.
$$
Since $\mathcal U(\germ c)$ is finite dimensional, 
$\mathcal W$ is finite dimensional as well, 
and hence it is a closed subspace of $\mathcal H$. 
It is easily seen that $\mathcal W$ is a $\mathbb Z_2$-graded, $\germ c$-invariant
(and hence $(C_0,\germ c)$-invariant) subspace of $\mathcal H$. Since 
$(\pi,\rho^\pi,\mathcal H)$ is 
irreducible, it follows that $\mathcal W=\mathcal H$. Therefore we have proved that
every irreducible unitary representation of $(C_0,\germ c)$ is finite dimensional.

Next observe that $\mathcal U(\germ c)/\langle Z-a\rangle$ is 
isomorphic (as a $\mathbb Z_2$-graded algebra) to a 
complex Clifford algebra.
It is a well-known result in the theory  of Clifford modules 
that up 
to parity change, 
a complex Clifford algebra
has a unique nontrivial finite 
dimensional irreducible $\mathbb Z_2$-graded representation. 
(See \cite[Chapter 5]{lawson} or \cite[Lemma 11]{vara}.)
If $\dim \germ c_1$ is odd, 
then the choice of $\mathbb Z_2$-grading does not matter, 
whereas if this dimension is even, then
parity change yields two non-isomorphic modules.
Conversely, fixing an $a>0$ and an irreducible $\mathbb Z_2$-graded
module $\mathcal K$ for the complex Clifford algebra
$\mathcal U(\germ c)/\langle Z-a\rangle$, 
one can obtain 
an irreducible unitary 
\footnote{Unitarity of this module follows from 
standard constructions of 
Clifford modules. See \cite{lawson} or 
\cite[Section 4.2]{vara}.}
representation 
$(\sigma_\mu,\rho^{\sigma_\mu},\mathcal K_\mu)$
of $(C_0,\germ c)$, where 
$\mathcal K_\mu=\mathcal K$ as a vector space and
$\mu:\germ c_0\to\mathbb R$ is an $\mathbb R$-linear functional
such that $\mu(Z)=a$ and 
\begin{equation}
\label{eqnabtmu}
\textrm{for every }W\in\germ c_0\textrm{ and }v\in \mathcal K_\mu,\ \ 
\sigma_\mu^\infty(W)v=\mu(W)\sqrt{-1}\,v.
\end{equation}
Note that the condition $a>0$ implies that
$\mu([V,V])>0$ for every $V\in\germ c_1$.

In conclusion, if $(\sigma_0,\rho^{\sigma_0},\mathcal K_0)$ denotes
the $(1|0)$-dimensional trivial representation of $(C_0,\germ c)$, then
we have proved the following statement.\footnote{Strictly speaking, 
there is ambiguity in the choice of 
$(\sigma_\mu,\rho^{\sigma_\mu},\mathcal K_\mu)$ up to parity change. 
However, for our purposes the choice of $\mathbb Z_2$-grading 
does not really matter, since
special induction commutes with parity change.}

\begin{proposition}
\label{propcl}
Let $(C_0,\germ c)$ be a super Lie group of Clifford type
and $(\pi,\rho^\pi,\mathcal H)$ be an irreducible unitary representation
of $(C_0,\germ c)$. Then there exists a unique 
$\mathbb R$-linear functional $\mu:\germ c_0\to\mathbb R$
satisfying $\mu([V,V])\geq 0$ for every $V\in\germ c_1$ such that
$$
(\pi,\rho^\pi,\mathcal H)\eqsim
(\sigma_\mu,\rho^{\sigma_\mu},\mathcal K_\mu).
$$ 
The representation $(\pi,\rho^\pi,\mathcal H)$ is 
trivial if and only if $\mu=0$.
\end{proposition}

\section{Realization as induced representations}
\label{inductionresult}
\subsection{Codimension-one induction}
\label{inductionresultfive}
Throughout this section $(N_0,\germ n)$ will be a reduced nilpotent 
super Lie group such 
that $\dim\germ n>1$ and $\dim\mathcal Z(\germ n)=1$. 
We assume that $(N_0,\germ n)$ is not
of Clifford type. Hence part (a) of Proposition \ref{kirillovlemma} holds for $\germ n$.
Let $\germ n'=\germ n'_0\oplus\germ n'_1$, $X$, $Y$, and $Z$ be 
as in part (a) of Proposition \ref{kirillovlemma}. 
Let $(N'_0,\germ n')$ be the sub super Lie group of $(N_0,\germ n)$ 
corresponding to $\germ n'$.

Let
$(\pi,\rho^\pi,\mathcal H)$ be an irreducible unitary representation of $(N_0,\germ n)$.
By \cite[Lemma 5]{vara}, there exists a real number $b\in\mathbb R$ 
such that for every $t\in\mathbb R$ and $v\in\mathcal H$, we have 
$$
\pi(\exp(tZ))v=e^{tb\sqrt{-1}}v.
$$
The main goal of this section is to prove the following.
\begin{proposition}
\label{induction-prop}
Suppose that the restriction of $\pi^\infty$ to $\mathcal Z(\germ n)$ 
is nontrivial, i.e.,
$b\neq 0$. Then 
there exists an irreducible unitary representation 
$(\sigma,\rho^\sigma,\mathcal K)$ of $(N'_0,\germ n')$ such that 
$$
\pi\simeq\ind_{(N'_0,\germ n')}^{(N_0,\germ n)}(\sigma,\rho^\sigma,\mathcal K).
$$
\end{proposition}
The rest of this section is 
devoted to the proof of Proposition \ref{induction-prop}.
The proof is inspired by that of \cite[Proposition 2.3.4]{corgr}, 
but there are
several crucial technical points that our proof deviates from
the argument given in \cite{corgr}. 

Set 
$\germ h=\spn_\mathbb R\{X,Y,Z\}$. 
Clearly, $\germ h$ 
is a Heisenberg Lie subalgebra of $\germ n_0$, corresponding to a 
Heisenberg Lie subgroup $H$ of $N_0$.

Fix an $i\in\{0,1\}$. The space $\mathcal H_i$ is an $N_0$-invariant subspace
of $\mathcal H$, and we will
denote this representation of $N_0$ by $(\pi_i,\mathcal H_i)$.
Let $\mathcal H_i^\infty$ be 
the space of smooth vectors of $(\pi_i,\mathcal H_i)$. 

From the proof of \cite[Proposition 2.3.4]{corgr}
it follows that 
$\pi_i\simeq\ind_{N'_0}^{N_0}\sigma_i$ where $\sigma_i$ is a unitary representation
of $N'_0$. For the reader's convenience, we give an 
outline of the argument.
By the Stone-von Neumann theorem 
in the form stated in \cite[2.2.9]{corgr}, 
there exist
a Hilbert space $\mathcal K_i$ and a linear isometry 
\begin{equation}
\label{isomisomisom}
S_i:\mathcal H_i\to L^2(\mathbb R,\mathcal K_i)
\end{equation}
such that $S_i$ intertwines the action of $H$, where the action of 
$H$ on 
$L^2(\mathbb R,\mathcal K_i)$ is 
given as follows : for every $s,t\in \mathbb R$ and $f\in L^2(\mathbb R,\mathcal K_i)$,
\begin{eqnarray*}
\big(\pi_i(\exp(tX))f\big)\,(s)&=&f(s+t)\\[0.5mm]
\big(\pi_i(\exp(tY))f\big)\,(s)&=&e^{bts\sqrt{-1}}f(s)\\[0.5mm]
\big(\pi_i(\exp(tZ))f\big)\,(s)&=&e^{tb\sqrt{-1}}f(s).\\[-3mm]
\end{eqnarray*}
Lemma 2.3.2 of \cite{corgr}
is still valid, and Lemma 2.3.1 of \cite{corgr} implies that
for every 
$g\in N'_0$, there exists a family $\{T_{g,t}\}_{t\in\mathbb R}$ of unitary
operators from $\mathcal K_i$ to $\mathcal K_i$ such that for every 
$f\in\mathcal H_i^\infty$, $g\in N_0'$, and $t\in\mathbb R$, 
we have
$$
\big(\pi_i(g)f\big)\,(t)=T_{g,t}(f(t)).
$$ 
The rest of the argument, i.e., showing
that the choice of 
\begin{equation}
\label{sigi}
\sigma_i(g)=T_{g,0}
\end{equation}
defines a unitary representation of $N'_0$
on $\mathcal K_i$, and that $\pi_i\simeq\ind_{N'_0}^{N_0}\sigma_i$, follows 
the proof of \cite[Proposition 2.3.4]{corgr} \emph{mutatis mutandis}.


Since $\pi_i\simeq\ind_{N'_0}^{N_0}\sigma_i$, the Hilbert 
space $\mathcal H_i$
can be realized as a space of functions from $N_0$ to $\mathcal K_i$
(see Section \ref{realization}). If we pick this realization of $\mathcal H_i$,
then the isometry $S_i$ of (\ref{isomisomisom}) 
is given by a simple formula which we now describe.
Let 
\begin{equation}
\label{prodductt}
L_0=\{\ \exp(tX)\ |\ t\in\mathbb R\ \}
\end{equation}
be the one-parameter subgroup of $N_0$
corresponding to $X$.
Then 
after normalizing the inner products, the map 
\begin{equation}
\label{sdefn}
S_i:\mathcal H_i\to L^2(\mathbb R,\mathcal K_i)
\end{equation}
is given by 
$$
S_if(t)=f(\exp(tX)).
$$

Let $\mathcal H_i^{\infty,c}$
denote the subspace of $\mathcal H_i^\infty$ consisting of functions
with compact support modulo $N'_0$.
Let $\cccs$ and $\cccc$ be the subspaces of $L^2(\mathbb R,\mathcal K_i)$
defined by
$$
\cccs=S_i\,\mathcal H_i^\infty\textrm{  and  }
\cccc=S_i\,\mathcal H_i^{\infty,c}. 
$$
Let $\mathcal K=\mathcal K_0\oplus\mathcal K_1$ be the orthogonal
direct sum of $\mathcal K_0$ and $\mathcal K_1$ and the isometry
$$
S:\mathcal H^{\infty,c}\to L^2(\mathbb R,\mathcal K)
$$ 
be defined as $S=S_0\oplus S_1$.
We also set
$$
\ccca=S\,\mathcal H^{\infty,c}\textrm{\ \ and\ \ }\cccsi=S\mathcal H^\infty.
$$
As usual, let $C^\infty(\mathbb R)$ denote the space of complex valued
smooth functions on $\mathbb R$, and $C_c^\infty(\mathbb R)$ 
denote the subspace of $C^\infty(\mathbb R)$ consisting of
functions with compact support. 

\begin{lemma}
\label{cpctsmooth}
Assume the above notation.
\begin{itemize}
\item[(a)] If $\phi\in C^\infty(\mathbb R)$ and $f\in\cccc$
then $\phi\,f\in\cccc$. 
\item[(b)] If $\phi\in C_c^\infty(\mathbb R)$ and $f\in\cccs$
then $\phi\,f\in\cccc$.
\end{itemize}
\end{lemma}
\begin{proof}
Let $L_0$ be defined as in (\ref{prodductt}).
Every element $n\in N_0$ can be written uniquely as
a product  
$n=n'\cdot l$ 
of an element
$n'\in N'_0$ and an element $l\in L_0$. 
Consider the function $\psi:N_0\to \mathbb C$ 
defined by $\psi(n)=\phi(l)$.
It is easily seen that
$\psi$ is smooth. To prove part (a) it suffices to show 
that if $h\in\mathcal H_i^{\infty,c}$, then $\psi\, h\in\mathcal H_i^{\infty,c}$,
and the latter inclusion follows from the description
of smooth vectors for induced
representations in
\cite[Theorem 5.1]{poulsen} or \cite[Theorem A.1.4]{corgr}. 
The proof of part (b) is similar.

\end{proof}

\begin{lemma}
\label{commuting}
Let $V\in\germ n_1$. If $\phi\in C^\infty_c(\mathbb R)$ and
$f\in\cccsi$ then
$$
\rho^\pi(V)(\phi f)=\phi\rho^\pi(V) f.
$$ 
\end{lemma}
\begin{proof}
Let $M_{\chi_a}:L^2(\mathbb R,\mathcal K)\to L^2(\mathbb R,\mathcal K)$ 
be the operator
of multiplication by
$\chi_a$, i.e., 
$$
\textrm{for every\ }h\in L^2(\mathbb R,\mathcal K)\textrm{ and }
t\in\mathbb R,\ \ (M_{\chi_a}h)\,(t)=\chi_a(t)h(t),
$$ 
where $\chi_a(t)=e^{at\sqrt{-1}}$. By Lemma \ref{cpctsmooth},
for every $i\in\{0,1\}$ we have
$$
M_{\chi_a}(\cccc)\subseteq \cccc.
$$
From $[Y,V]=0$ it follows that for every $a\in\mathbb R$,
\begin{equation}
\label{starstar}
\rho^\pi(V)M_{\chi_a}=M_{\chi_a}\rho^\pi(V).
\end{equation} 

Choose a sequence $\{\phi_n\}_{n=1}^\infty$ of elements of 
$\spn_\mathbb C\{\,\chi_a\,|\,a\in\mathbb R\,\}$ such that 
\begin{equation}
\label{unifconv}
\lim_{n\to\infty}\phi_nf=\phi f
\textrm{\ \ and\ \ }
\lim_{n\to\infty}\phi_n\rho^\pi(V)f=\phi\rho^\pi(V)f,
\end{equation}
where the convergences are in $L^2(\mathbb R,\mathcal K)$.
The sequence $\{\phi_n\}_{n=1}^{\infty}$ can be found as follows. 
By the Stone-Weierstrass theorem, for every positive integer 
$n$ one can choose
a function 
$$
\phi_n\in\spn_\mathbb C\{\ \chi_a\ |\ a\in\mathbb R\ \}
$$ 
which is periodic with period 
$2n$, and
$$
\max_{-n\leq t\leq n}\{\,|\phi(t)-\phi_n(t)|\,\}\leq{1\over n}.
$$
Since elements of $\cccsi$ are smooth vectors for the action
of the Heisenberg group $H$, they are Schwartz functions from $\mathbb R$
to $\mathcal K$. Suppose $n$ is large enough such that 
$$
\mathrm{Supp}(\phi)\subset\{x\in\mathbb R\,|\,-n\leq x\leq n\}.
$$
Now we have
\begin{eqnarray*}
||(\phi-\phi_n)f||^2&\leq&\int_{-n}^n||(\phi_n(t)-\phi(t))f(t)||^2dt+\int_{|t|>n}||(\phi_n(t)-\phi(t))f(t)||^2dt\\
&\leq&{1\over n^2}\times 2n\times \max_{|t|\leq n}\{||f(t)||^2\}
+
({1\over n}+\max_{t\in\mathbb R}\{|\phi(t)|\})^2
\int_{|t|>n}||f(t)||^2dt.
\end{eqnarray*}
Since $f\in\cccsi$, when $n$ grows to infinity the last line above converges to zero.
Since $\rho^\pi(V)f\in\cccsi$,
the same reasoning applies to $\rho^\pi(V)f$ instead of $f$ as well.

Since 
$\phi_n\in\spn_\mathbb C\{\,\chi_a\,|\,a\in\mathbb R\,\}$, it follows 
from (\ref{starstar}) that 
$$
\rho^\pi(V)\phi_n f=\phi_n\rho^\pi(V) f.
$$ 
The operator $\rho^\pi(V)$ is symmetric, 
hence it is closable (see \cite[p. 316]{conway}). In particular,
from (\ref{unifconv})  
and the fact that $\phi f\in D(\rho^\pi(V))$ it follows that
$$
\rho^\pi(V)(\phi\, f)=\phi\,\rho^\pi(V)f.
$$

\end{proof}
Consider the map $\Psi:N'_0\times \mathbb R\to N_0$ defined as
\begin{equation}
\label{defpsi}
\Psi(n',s)=n'\cdot\exp(sX).
\end{equation} 
The map $\Psi$ is bijective and the Campbell-Baker-Hausdorff
formula implies that it is smooth. Moreover, for every 
$\mathbf x=(n',s)\in N'_0\times \mathbb R$, if by means of the left action of $N_0$
we identify the tangent
spaces at $\mathbf x$ and $\Psi(\mathbf x)$ with 
$\germ n'_0\oplus \mathbb R$ and
$\germ n_0$, then the derivative 
$$
\mathbf D\Psi(\mathbf x):\germ n'_0\oplus \mathbb R\to \germ n_0
$$ of $\Psi$ at $\mathbf x$ 
is given by the following formula.
$$
\textrm{For every } (Q,t)\in\germ n_0'\oplus\mathbb R,\ \ 
\mathbf D\Psi(\mathbf x)(Q,t)=(\exp(-sX)\cdot Q,t).
$$
Note that $\exp(-sX)\cdot Q\in\germ n'_0$ because 
$N'_0$ is a normal 
subgroup of $N_0$ (see \cite[Lemma 1.1.8]{corgr}).

From the formula for $\mathbf D\Psi$ it follows 
immediately 
that $\mathbf D\Psi(\mathbf x)$ 
is invertible at every $\mathbf x\in N_0'\times\mathbb R$.  
Hence the inverse mapping
theorem implies that $\Psi^{-1}$ is smooth. 
Consequently, 
a function $f:N_0\to \mathcal K_i$ is smooth if and only if
$f\circ\Psi:N'_0\times\mathbb R\to\mathcal K_i$ is smooth. 

\begin{lemma}
\label{indirectiont}
Let $f:\mathbb R\to \mathcal K_i$ be a smooth function
such that $f(0)=0$. 
Set 
$$
g(t)=\left\{
\begin{array}{ll}
{f(t)\over t}\ \ \ &\textrm{if\ }t\neq 0,\\[3mm]
f'(0)&\textrm{otherwise.}
\end{array}
\right.
$$
Then $g$ is a smooth function as well, and 
$g^{(n)}(0)={f^{(n+1)}(0)\over n+1}$.
\end{lemma}
\begin{proof}
For $t\neq 0$, the lemma is trivial.
We prove the Lemma for $t=0$ by induction on $n$, in each step 
proving
that $g^{(n)}(0)={f^{(n+1)}(0)\over n+1}$.
For $n=0$ the latter statement is 
obvious. We will assume that the statement is true 
for some $n$, 
and prove it for $n+1$. 
If $t\neq 0$, then 
$$
g^{(n)}(t)=\sum_{k=0}^n\Big({n\atop k}\Big)(-1)^{n-k}(n-k)!t^{-n+k-1}f^{(k)}(t).
$$
Therefore

\begin{eqnarray*}
\lim_{t\to 0}{g^{(n)}(t)-g^{(n)}(0)\over t}&=&
\lim_{t\to 0}
{\displaystyle\Big(\sum_{k=0}^n\Big({n\atop k}\Big)(-1)^{n-k}(n-k)!\,t^k
f^{(k)}(t)
\Big) - 
{t^{n+1}f^{(n+1)}(0)\over n+1}
\over t^{n+2}}
\end{eqnarray*}
which is a limit of the form 
$$
\lim_{t\to 0}{h_1(t)\over h_2(t)}
$$
where $h_1$ and $h_2$ are continuously
differentiable functions and $h_1(0)=h_2(0)=0$.
It follows that the above limit is equal to
$$
\lim_{t\to 0}{h'_1(t)\over h'_2(t)}
$$  
in case the latter limit exists. 
But we have

\begin{eqnarray*}
h'_1(t)&=&
n!\sum_{k=0}^n(-1)^{n-k}f^{(k+1)}(t){t^k\over k!}\\
&+&\Big(n!\sum_{k=1}^n(-1)^{n-k}f^{(k)}(t){t^{k-1}\over (k-1)!}\Big)
-t^nf^{(n+1)}(0)\\
&=&t^nf^{(n+1)}(t)-t^nf^{(n+1)}(0)
\end{eqnarray*}
while $h'_2(t)=(n+2)t^{n+1}$.
In conclusion, we have
\begin{eqnarray*}
\lim_{t\to 0}{h'_1(t)\over h'_2(t)}=\lim_{t\to 0}{(t^nf^{(n+1)}(t)-t^nf^{(n+1)}(0))
\over (n+2)t^{n+1}}=
{f^{(n+2)}(0)\over n+2}
\end{eqnarray*}
which implies that
$g^{(n+1)}(0)={f^{(n+2)}(0)\over n+2}$.
\end{proof}
\begin{lemma}
\label{elemanal}
Let $q$ be a positive integer and $f:\mathbb R^q\times\mathbb R\to\mathcal K_i$ be a smooth
function such that  
for every $x\in\mathbb R^q$, we have $f(x,0)=0$. 
Define
$$
g(x,t)=
\left\{
\begin{array}{ll}
{f(x,t)\over t}&\textrm{if }t\neq 0,\\[3mm]
{\partial f\over\partial t}(x,0)&\textrm{otherwise.}
\end{array}
\right.
$$
Then $g(x,t)$ is smooth, and indeed 
\begin{equation}
\label{egalite}
{\partial^ng\over\partial t^n}(x,0)=
{\displaystyle{\partial^{n+1}f\over\partial t^{n+1}}(x,0)\over n+1}.
\end{equation}
\end{lemma}
\begin{proof}
That $g(x,t)$ is smooth when $t\neq 0$ is trivial.
From Lemma \ref{indirectiont} it follows that for every integer
$n\geq 0$ and every $x\in\mathbb R^q$, 
${\partial^ng\over\partial t^n}(x,0)$ exists and 
equality (\ref{egalite}) holds. 

Every differential operator in $x_1,...,x_q,t$
is a linear combination of operators $\D$ of the form  
$$
\D={\partial^{a_1}\over \partial x_{i_1}^{a_1}}
{\partial^{b_1}\over \partial t^{b_1}}
{\partial^{a_2}\over\partial x_{i_2}^{a_2}}
{\partial^{b_2}\over \partial t^{b_2}}
\cdots
{\partial^{a_k}\over\partial x_{i_k}^{a_k}}
{\partial^{b_k}\over \partial t^{b_k}}
$$
where $i_1,...,i_k\in\{1,...,q\}$ and $a_1,...,a_k,b_1,...,b_k\in\{0,1\}$.
(For example, if $k=3$, $a_1=a_3=1$, $a_2=0$, $b_1=b_2=1$, $b_3=0$, $i_1=3$,
$i_3=2$, and $1\leq i_2\leq q$ 
then $\D={\partial\over\partial x_3}{\partial^2\over\partial t^2}
{\partial\over\partial x_2}$.) In order to complete the proof of  
the lemma, it suffices
to show that for every such $\D$ and every $x\in\mathbb R^q$,
the partial derivative $\D g(x,0)$ exists.

For every $1\leq i\leq q$ and $n\geq 0$ we have 
$$
{\partial\over \partial x_i}{\partial^n g\over\partial t^n}(x,t)=
\left\{
\begin{array}{ll}
{\partial^n\over \partial t^n}\big(
{1\over t}{\partial f\over \partial x_i}
\big)(x,t)
&\textrm{if }t\neq 0,\\[3mm]
{1\over n+1}
\big(
{\partial^{n+1}\over \partial t^{n+1}}
{\partial f\over \partial x_i}
\big)(x,0)
&\textrm{otherwise.}
\end{array}
\right.
$$
Thus for every $(x,t)\in \mathbb R^q\times \mathbb R$ 
the partial derivative 
${\partial\over\partial x_i}{\partial^ng\over\partial t^n}(x,t)$ exists and
if we set
$$
g_1(x,t)=
\left\{
\begin{array}{ll}
{{\partial f\over \partial x_i}(x,t)\over t}
&\textrm{if }t\neq 0,\\[3mm]
{\partial\over\partial t}{\partial f\over \partial x_i}(x,0)
&\textrm{otherwise}
\end{array}
\right.
$$
then we have 
$$
{\partial\over\partial x_i}{\partial^ng\over\partial t^n}(x,t)=
{\partial^ng_1\over\partial t^n}(x,t).
$$
\vspace{1mm}

\noindent Note that Lemma \ref{indirectiont} implies that
${\partial^ng_1\over\partial t^n}(x,t)$ exists for every $n\geq 0$. 

By repeating the above argument one can show that
$\D g(x,t)$ exists and is equal to 
$$
{\partial^{b_1+\cdots b_k}\over\partial t^{b_1+\cdots+b_k}}g_2(x,t)
$$ 
where
$$
g_2(x,t)=
\left\{
\begin{array}{ll}
{\displaystyle 1\over\displaystyle t}
\Big(
{\displaystyle
\partial^{a_1+\cdots+a_k} 
\over 
\displaystyle
\partial x_{i_1}^{a_1}\partial x_{i_2}^{a_2}\cdots
\partial x_{i_k}^{a_k}
}
f(x,t)
\Big)
&\textrm{if }t\neq 0,\\[6mm]
{\displaystyle\partial
\over
\displaystyle\partial t}
\Big(
{\displaystyle
\partial^{a_1+\cdots+a_k} \over 
\displaystyle
\partial x_{i_1}^{a_1}\partial x_{i_2}^{a_2}\cdots
\partial x_{i_k}^{a_k}
}f\Big)(x,0)
&\textrm{otherwise.}
\end{array}
\right.
$$
The existence of 
${\partial^{b_1+\cdots b_k}\over\partial t^{b_1+\cdots+b_k}}g_2(x,t)$
follows from Lemma \ref{indirectiont}.

\end{proof}

\begin{lemma}
\label{delicacy}
Let $f\in\cccc$ satisfy $f(0)=0$. Then there exists a function 
$g\in\cccc$
such that for every $t\in\mathbb R$ we have $f(t)=tg(t)$.
\end{lemma}
\begin{proof}
Let $f=S_i\,h$ where $S_i$ is the operator defined in (\ref{sdefn})
and $h\in\mathcal H_i^{\infty,c}$. Set $h_1=h\circ\Psi$ where
$\Psi:N'_0\times \mathbb R\to  N_0$ is the map defined in (\ref{defpsi}).
For every $(n',t)\in N_0'\times\mathbb R$
we have 
$$
h_1(n',t)=t\,h_2(n',t)
$$ 
where $h_2:N_0'\times \mathbb R\to\mathcal K_i$
is defined as follows.
$$h_2(n',t)=\left
\{\begin{array}{ll}
{h_1(n',t)\over t}&\textrm{if }t\neq 0,\\[3mm]
{\partial h_1\over\partial t}(n',0)&\textrm{otherwise.}
\end{array}
\right.
$$
Lemma \ref{elemanal} implies that $h_2$ is smooth. From the description
of smooth vectors for induced representations given in \cite[Theorem 5.1]{poulsen}
or \cite[Theorem A.1.4]{corgr} it follows that the function 
$h_3=h_2\circ\Psi^{-1}$ belongs to $\mathcal H_i^{\infty,c}$. 
To complete the proof, we set $g=S_i\,h_3$.

\end{proof}

\begin{lemma}
\label{welldef}
Let $V\in\germ n_1$. Suppose that $f\in\cccc$ satisfies
$f(0)=0$. Then 
$$
\big(\rho^\pi(V)f\big)\,(0)=0.
$$
\end{lemma}
\begin{proof}
By Lemma \ref{delicacy} we have $f(t)=t\,g(t)$ where $g\in\cccc$.
Since $g$ has compact support, it is not hard to see that there exists
a function $\psi\in C^\infty_c(\mathbb R)$ such that $\psi(t)=1$ for 
$t\in\mathrm{Supp}(g)\cup\{0\}$. It follows that $f=\psi f$. 
Set $\psi_1(t)=t\,\psi(t)$.
By Lemma \ref{commuting} we have 
$$
Tf=T\psi f=T(\psi_1g)=\psi_1Tg,
$$ hence $(Tf)(0)=\psi_1(0)\big((Tg)(0)\big)=0$,
which completes the proof.

\end{proof}
\vspace{6mm}

\noindent{\bf Remark.} 1. 
Let $V\in\germ n_1$ and suppose that $f\in\cccc$ 
satisfies $f(t_\circ)=0$ for some
$t_\circ\in\mathbb R$. 
Then $\big(\pi_i(\exp(t_\circ X))f\big)(0)=0$, hence
by Lemma \ref{welldef}
\begin{eqnarray}
\label{eqnabouttzero}
\notag
\big(\rho^{\pi_i}(V)f\big)(t_\circ)&=&
\big(\pi_i(\exp(t_\circ X))\rho^{\pi_i}(V)f\big)(0)\\
&=&\Big(\rho^{\pi_i}\big(\exp(t_\circ X)\cdot V\big)\pi_i(\exp(t_\circ X))f\Big)(0)
=0
\end{eqnarray}
\vspace{-1mm}

\noindent 2. An immediate consequence of (\ref{eqnabouttzero}) is that
$\rho^\pi(V)\ccca\subseteq\ccca$.\\[1mm]

For every $V\in\germ n_1$ we define a family of linear
operators 
$$
T_{V,t}:\mathcal K_1^\infty\oplus\mathcal K_2^\infty\to 
\mathcal K_1^\infty\oplus\mathcal K_2^\infty
$$ 
as follows. For every $i\in\{0,1\}$, $t_\circ\in\mathbb R$, and 
$v\in\mathcal K_i^\infty$, choose an $f\in\cccc$ such that 
$f(t_\circ)=v$.
For instance, one can fix $\varphi\in C_c^\infty(\mathbb R)$ such that 
$\varphi(t_0)=1$, and take $f=S_i(h\circ \Psi^{-1})$ where
$h:N'_0\times\mathbb R\to \mathcal K_i$ is given by 
$$
h(n',t)=\varphi(t)\sigma_i(n')v.
$$
Now set 
$$
T_{V,t}v=(\rho^\pi(V)f)(t).
$$
Lemma 
\ref{welldef} and the remark after this lemma 
imply that the operators $T_{V,t}$ are well defined.
Since $\rho^\pi(V)$ is odd, it follows that the $T_{V,t}$'s are 
odd operators. 

We now
set $\sigma=\sigma_0\oplus\sigma_1$, and
for any $W\in\germ n_1$ define 
$\rho^\sigma(W):\mathcal K^\infty\to\mathcal K^\infty$ 
by 
$\rho^\sigma(W)v=T_{W,0}v$.
Our next task is to verify that the triple $(\sigma,\rho^\sigma,\mathcal K)$ 
satisfies the conditions of Definition \ref{unirep}. Linearity of
$\rho^\sigma$ and condition (a) 
of Definition \ref{unirep} are obvious. Next we prove 
that for every $W\in\germ n_1$ the
operator $\rho^\sigma(W)$ is symmetric.
Suppose, on the contrary, that $\rho^\sigma(W)$ is not 
symmetric, and
let $v,w\in\mathcal K^\infty$ such that 
$$
\langle\rho^\sigma(W)v,w\rangle\neq\langle v,\rho^\sigma(W)w\rangle.
$$
Choose $\varphi\in C_c^\infty(\mathbb R)$ such that $\phi(0)=1$, and consider
two functions 
$$
f_v,f_w:N'_0\times \mathbb R\to \mathcal K
$$ 
defined by
$$
f_v(n',t)=\varphi(t)\sigma(n')v\textrm{\ \ \ and\ \ \  }
f_w(n',t)=\varphi(t)\sigma(n')w.
$$
The functions $f_v\circ\Psi^{-1}$ and $f_w\circ\Psi^{-1}$ belong to
$\mathcal H^{\infty,c}$.
Let $g_v,g_w\in\ccca$ be defined by
$$
g_v=S(f_v\circ\Psi^{-1}) \textrm{\ \ \  and\ \ \ } 
g_w=S(f_w\circ\Psi^{-1}).
$$
It is readily seen that 
$$
\big(\rho^\pi(W)g_v\big)\,(t)=T_{\exp(tX)\cdot W,0}\big(g_v(t)\big).
$$
If $\{W,W_1,...,W_r\}$ is a basis containing $W$ for $\germ n_1$, then 
\begin{eqnarray}
\label{analytaa}
\exp(tX)\cdot W=\gamma_0(t)W+\sum_{i=1}^r\gamma_i(t)W_i
\end{eqnarray}
where for every $0\leq i\leq r$ the function 
$\gamma_i(t):\mathbb R\to\mathbb R $ is smooth. Moreover,  
$$
\lim_{t\to 0}\gamma_0(t)=1 \textrm{\,\, and for every }1\leq i\leq r,\ 
\lim_{t\to 0}\gamma_i(t)=0.
$$
Now 
\begin{eqnarray*}
\big(\rho^\pi(W)g_v\big)(t)&=&
T_{\exp(tX)\cdot W,0}(g_v(t))\\
&=&\gamma_0(t)T_{W,0}(g_v(t))
+\sum_{i=1}^r\gamma_i(t)T_{W_i,0}(g_v(t))\\
&=&\varphi(t)
\left(\gamma_0(t)T_{W,0}v+\sum_{i=1}^r\gamma_i(t)T_{W_i,0}v\right).
\end{eqnarray*}
But
\begin{equation}
\label{aboutt-1}
\lim_{t\to 0}
\Big(\gamma_0(t)T_{W,0}v+\sum_{i=1}^r\gamma_i(t)T_{W_i,0}v\Big)=
T_{W,0}v=\rho^\sigma(W)v
\end{equation}
and
\begin{eqnarray}
\label{aboutt-2}
\langle\rho^\pi(W)g_v,g_w\rangle&=&
\int_{-\infty}^\infty\langle\,\big(\rho^\pi(W)g_v\big)(t),g_w(t)\,\rangle\,dt\\
&=&\int_{-\infty}^\infty\varphi(t)^2
\langle\,\Big(\gamma_0(t)T_{W,0}v+\sum_{i=1}^r\gamma_i(t)T_{W_i,0}v\Big)
,w \,\rangle\,dt.
\notag
\end{eqnarray}
Similarly we have 
\begin{equation}
\label{aboutt-3}
\lim_{t\to 0}
\Big(\gamma_0(t)T_{W,0}w+\sum_{i=1}^r\gamma_i(t)T_{W_i,0}w\Big)=
T_{W,0}w=\rho^\sigma(W)w
\end{equation}
and 
\begin{eqnarray}
\label{aboutt-4}
\langle g_v,\rho^\pi(W) g_w\rangle&=&
\int_{-\infty}^\infty\langle\,g_v(t),\big(\rho^\pi(W)g_w\big)(t)\,\rangle\,dt\\
&=&\int_{-\infty}^\infty\varphi(t)^2
\langle\,v,
\Big(\gamma_0(t)T_{W,0}w+\sum_{i=1}^r\gamma_i(t)T_{W_i,0}w\Big) \,\rangle\,dt .
\notag
\end{eqnarray}
From (\ref{aboutt-1}), (\ref{aboutt-2}), 
(\ref{aboutt-3}) and (\ref{aboutt-4}) 
it follows that if $\mathrm{Supp}(\varphi)$ is 
small enough, then 
$$
\langle \rho^\pi(W)g_v,g_w\rangle\,\neq\,\langle g_v,\rho^\pi(W)g_w\rangle
$$
which contradicts the fact that $\rho^\pi(W)$ is symmetric.

Condition (c) of Definition \ref{unirep} can be verified as follows.
Let $V,W\in\germ g_1$ and $v\in \mathcal K^\infty$. Choose an $f\in\ccca$
such that $f(0)=v$. We have 
\begin{eqnarray*}
\big(
\rho^\sigma(V)\rho^\sigma(W)\!+\!\rho^\sigma(V)\rho^\sigma(W)
\big)v
&=&
\Big(\rho^\pi(V)\big(\rho^\pi(W)f\big)\Big)(0)
\!+\!
\Big(\rho^\pi(V)\big(\rho^\pi(W)f\big)\Big)(0)\\
&=&
-\sqrt{-1}\big(\pi^\infty([V,W])f\big)(0).
\end{eqnarray*}
Since $[V,W]\in\germ n_0'$, from
(\ref{sigi})
it follows that 
$$
\big(\pi^\infty([V,W])\big)f(0)=\sigma^\infty([V,W])\big(f(0)\big)=
\sigma^\infty([V,W])v
$$
which completes the proof of condition (c).

Finally, we prove condition (d) of Definition \ref{unirep}.
Let $V\in\germ n_1$, $g\in N'_0$, and $w\in \mathcal K^\infty$.
Let $f\in\ccca$ be such that $f(0)=w$.
Using (\ref{sigi}) we have
\begin{eqnarray*}
\rho^\sigma(g\cdot V)w&=&\big(\rho^\pi(g\cdot V)f\big)(0)\\
&=&
\big(\pi(g)\rho^\pi(V)\pi(g)^{-1}f \big)(0)\\
&=&\sigma(g)\Big(\big(\rho^\pi(V)\pi(g^{-1})f\big)(0)\Big)\\
&=&\sigma(g)
\Big(
\rho^\sigma(V)\Big(\big(\pi(g^{-1})f\big)(0)\Big)
\Big)\\
&=&\sigma(g)\Big(\rho^\sigma(V)\Big(\sigma(g^{-1})\big(f(0)\big)\Big)\Big)\\
&=&\sigma(g)\rho^\sigma(V)\sigma(g^{-1})w.
\end{eqnarray*}

To finish the proof of Proposition \ref{induction-prop}, note that 
the unitary representations 
$(\pi,\rho^\pi,\mathcal H)$ and 
$$
\ind_{(N'_0,\germ n')}^{(N_0,\germ n)}(\sigma,\rho^\sigma,\mathcal K)
$$
are identical on $\mathcal H^{\infty,c}$. 
This follows from the fact that for every 
$V\in\germ n_1$, $t\in\mathbb R$, and $f\in\ccca$ we have 
$$
(\rho^\pi(V)f\,)(t)=\rho^\sigma(\exp(tX)\cdot V)\big(f(t)\big).
$$
Consequently,
Proposition \ref{varadarajan} implies that these representations are unitarily 
equivalent.
Since $(\pi,\rho^\pi,\mathcal H)$
is assumed to be irreducible, it follows that 
$(\sigma,\rho^\sigma,\mathcal K)$
is irreducible as well. 


\subsection{Stone-von Neumann theorem for Heisenberg-Clifford supergroups}
\label{svngeneral}

In this section we show how to use Proposition \ref{induction-prop} 
to prove
a generalization of the Stone-von Neumann theorem for Heisenberg-Clifford
super Lie groups.

Let $(N_0,\germ n)$ be a Heisenberg-Clifford super Lie group
(see Section \ref{exampcliff}). For every $P,Q\in\germ n_1$,
the value of the bracket $[P,Q]$ lies in $\germ n_0=\mathbb R$,
and hence can be thought of as a real number.
Consider the symmetric
bilinear form 
$$
\mathrm B:\germ n_1\times\germ n_1\to \mathbb R
$$ 
defined by 
$\mathrm B(P,Q)=[P,Q]$.

Let $(\pi,\rho^\pi,\mathcal H)$ be 
an irreducible unitary representation
of $(N_0,\germ n)$. 
By \cite[Lemma 5]{vara} the action of $\mathcal Z(N_0)$ is 
via multiplication by a unitary character 
$$
\chi:\mathcal Z(N_0)\to \mathbb C^\times.
$$
When $\mathrm B$ is a definite form, we say that the 
character $\chi$ \emph{agrees with} $\mathrm B$ 
if there exists a positive real number $c$ such that 
for every $P\in\germ n_1$,
$$
\chi([P,P])=e^{cB(P,P)\sqrt{-1}}.
$$

If $\chi$ is the trivial character, then it is easily seen 
that the sub super Lie group
$(\mathcal Z(N_0),[\germ n_1,\germ n_1]\oplus\germ n_1)$
of $(N_0,\germ n)$ belongs to the kernel of $(\pi,\rho^\pi,\mathcal H)$.
Consequently, $(\pi,\rho^\pi,\mathcal H)$ yields an irreducible representation
of the abelian Lie group $N_0/\mathcal Z(N_0)$. 
It follows that 
$(\pi,\rho^\pi,\mathcal H)$ is a one-dimensional representation obtained
from a unitary character of $N_0/\mathcal Z(N_0)$.

If $\chi$ is not trivial, then we have the following
result.

\begin{theorem}
\label{generstonevn}
Suppose that the unitary character 
$\chi:\mathcal Z(N_0)\to \mathbb C^\times$ is nontrivial.

\begin{itemize}
\item[(a)] If $\mathrm B$ is an indefinite form, then there are no 
irreducible unitary representations of $(N_0,\germ n)$ with
central character $\chi$.
\item[(b)] Suppose that $\mathrm B$ is a definite form.
If $\chi$ agrees with
$\mathrm B$, then up to
unitary equivalence and parity change
there exists a unique irreducible unitary representation of $(N_0,\germ n)$
with central character $\chi$. If $\chi$ does not agree with $\mathrm B$,
then such a unitary representation does not exist.
\end{itemize}
\end{theorem}
\begin{proof}
Part (a) follows from the fact that $\germ a[\germ n]=[\germ n_1,\germ n_1]
\oplus\germ n_1$.
Part (b) is proved by induction on the dimension of $N_0$ as follows.
Let 
$$
\{Z,X_1,X_2,...,X_m,Y_1,...,Y_m,V_1,..,V_n\}
$$
be the basis of $\germ n$ given in
(\ref{basisiss}), and 
$(\pi,\rho^\pi,\mathcal H)$
be an irreducible unitary representation of $(N_0,\germ n)$ 
with central character $\chi$.
By Proposition
\ref{induction-prop} we have
$$(\pi,\rho^\pi,\mathcal H)=
\ind_{(N'_0,\germ n')}^{(N_0,\germ n)}(\sigma,\rho^\sigma,\mathcal K)
$$
where 
$$
\germ n'=\spn_\mathbb R\{Z,X_2,...,X_m,Y_1,...,Y_m,V_1,..,V_n\}.
$$
Moreover, from the proof of Proposition
\ref{induction-prop} it follows that $\sigma^\infty(Y_1)=0$.
Therefore $(\sigma,\rho^\sigma,\mathcal K)$ factors through
a representation of a Heisenberg-Clifford super Lie group
$(N_0'',\germ n'')$ where 
$\germ n''=\germ n'/\germ r$ and $\germ r=\spn_\mathbb R\{Y_1\}$.
Since $\dim N_0''<\dim N_0$, the proof 
is completed by induction on $\dim N_0$. Details are left
to the reader.

\end{proof}
\noindent{\bf Remark.} Suppose that $\chi$ is nontrivial, $\mathrm B$ is 
definite, and $\chi$ agrees with $\mathrm B$.
Then part (b) of Theorem \ref{generstonevn}
can be refined slightly as follows. When 
$\dim \germ n_1$ is even, there exist two 
irreducible unitary
representations which are not unitarily 
equivalent. However, when $\dim\germ n_1$ is
odd, we obtain a unique such representation up to unitary equivalence.
Indeed the restriction to 
$(\mathcal Z(N_0),[\germ n_1,\germ n_1]\oplus\germ n_1)$
of such a representation is a countable direct sum of 
modules for a complex Clifford algebra, 
and when $\dim \germ n_1$ is even there
are two nonisomorphic such modules \cite[Chapter 5]{lawson}. 
The details are left to the reader.

\section{Polarizing systems and main theorems}
\label{mainressec}
\subsection{Polarizing systems}
\label{mainressecsix}
Throughout this 
section $(N_0,\germ n)$ is a (not necessarily reduced)
nilpotent super Lie group. 

\begin{definition}
\label{polarizing}
A polarizing system of $(N_0,\germ n)$ is a 6-tuple 
$(M_0,\germ m,\Phi,C_0,\germ c,\lambda)$ 
where
\begin{itemize}
\item[(a)] $(M_0,\germ m)$ is a special 
sub super Lie group of $(N_0,\germ n)$. 
\item[(b)] $\lambda:\germ n_0\to \mathbb R$ 
is an $\mathbb R$-linear functional
and $\germ m_0$ is a  
polarizing subalgebra of $\germ n_0$ corresponding to $\lambda$. 
\item[(c)] $(C_0,\germ c)$ is a super Lie group of Clifford type 
and $\Phi$ is a surjective
homomorphism 
$$
\Phi:(M_0,\germ m)\to (C_0,\germ c).
$$
\item[(d)] $\germ m_0\cap\ker\Phi=\germ m_0\cap\ker \lambda$.
\end{itemize}
\end{definition}
Let $(M_0,\germ m,\Phi,C_0,\germ c,\lambda)$ be a polarizing system 
of $(N_0,\germ n)$ and $(\sigma_\mu,\rho^{\sigma_\mu},\mathcal K_\mu)$ be 
the irreducible unitary representation of $(C_0,\germ c)$ 
associated to a linear functional $\mu:\germ c_0\to\mathbb R$
(see Section \ref{unitclifford}).
One can compose $(\sigma_\mu,\rho^{\sigma_\mu},\mathcal K_\mu)$
with the map 
$$
\Phi:(M_0,\germ m)\to(C_0,\germ c)
$$ 
and obtain an irreducible unitary representation 
$(\sigma_\mu\circ\Phi,\rho^{\sigma_\mu\circ\Phi},\mathcal K_\mu)$ of $(M_0,\germ m)$.
The representation $(\sigma_\mu,\rho^{\sigma_\mu},\mathcal K_\mu)$
is said to be \emph{consistent} with the polarizing system
if
\begin{equation}
\label{compathadi}
\textrm{for every }W\in\germ m_0,\ \ \lambda(W)=\mu\circ\Phi(W).
\end{equation}
We will see below that consistent representations play
a special role in the classification of irreducible unitary representations.
\begin{theorem}
\label{thmindpol}
Let $(\pi,\rho^\pi,\mathcal H)$ be an irreducible unitary representation of
a nilpotent super Lie group $(N_0,\germ n)$. 
\begin{itemize}
\item[(a)]
There exists a polarizing system $(M_0,\germ m,\Phi,C_0,\germ c,\lambda)$
and an irreducible unitary representation $(\sigma_\mu,\rho^{\sigma_\mu},\mathcal K_\mu)$
of $(C_0,\germ c)$ 
which is consistent with
$$
(M_0,\germ m,\Phi,C_0,\germ c,\lambda)
$$ 
such that
\begin{equation}
\label{whatisit}
(\pi,\rho^\pi,\mathcal H)=
\ind_{(M_0,\germ m)}^{(N_0,\germ n)}
(\sigma_\mu\circ\Phi,\rho^{\sigma_\mu\circ\Phi},\mathcal K_\mu).
\end{equation}
\item[(b)]
Suppose that $(M'_0,\germ m',\Phi',C'_0,\germ c',\lambda')$ is another polarizing system 
and 
$$
(\sigma_{\mu'},\rho^{\sigma_{\mu'}},\mathcal K_{\mu'})
$$ 
is a representation
of 
$(C'_0,\germ c')$
consistent with 
$(M'_0,\germ m',\Phi',C_0',\germ c',\lambda')$
such that
$$
(\pi,\rho^\pi,\mathcal H)=
\ind_{(M'_0,\germ m')}^{(N_0,\germ n)}
(\sigma_{\mu'}\circ\Phi,\rho^{\sigma_{\mu'}\circ\Phi},\mathcal K_{\mu'}).
$$
Then
there exists an $n\in N_0$ such that 
$$
\lambda'=\mathrm{Ad}^*(n)(\lambda).
$$
Moreover, the super Lie groups $(C_0,\germ c)$ and $(C'_0,\germ c')$
are isomorphic.
\end{itemize}
\end{theorem}
\begin{proof}
Part (a) is proved by induction on $\dim\germ n$. There are three cases to consider:

{\it Case I : $(N_0,\germ n)$ is not
reduced.} In this case $(\pi,\rho^\pi,\mathcal H)$ factors through 
the reduced form $(\overline N_0,\overline{\germ n})$ of $(N_0,\germ n)$,
and $\dim \overline{\germ n}<\dim \germ n$. Let us denote this 
representation of $(\overline N_0,\overline{\germ n})$ by 
$(\overline\pi,\rho^{\overline\pi},\mathcal H)$. By induction
hypothesis, 
there exists a polarizing system 
$(\overline M_0,\overline{\germ m},\overline\Phi,\overline C_0,
\overline{\germ c},\overline\lambda)$
of 
$(\overline N_0,\overline{\germ n})$ 
and a representation 
$(\sigma_{\mu},\rho^{\sigma_{\mu}},\mathcal K_{\mu})$
of $(\overline C_0,\overline{\germ c})$
which is consistent with
$(\overline M_0,\overline{\germ m},\overline\Phi,\overline C_0,
\overline{\germ c},\overline\lambda)$
such that 
$$
(\overline\pi,\rho^{\overline\pi},\mathcal H)=
\ind_{(\overline M_0,\overline{\germ m})}^{(\overline N_0,\overline{\germ n})}
(\sigma_{\mu}\circ\overline\Phi,\rho^{\sigma_{\mu}
\circ\overline\Phi},\mathcal K_{\mu}).
$$
Let 
$\mathsf q:(N_0,\germ n)\to(\overline N_0,\overline{\germ n})$ 
be the quotient map and
set
$$
(M_0,\germ m,\Phi,C_0,\germ c,\lambda)=(\mathsf q^{-1}(\overline M_0),
\mathsf q^{-1}(\overline{\germ m}),\overline\Phi\circ \mathsf q,
\overline C_0,\overline{\germ c},
\overline\lambda\circ \mathsf q).
$$ 
It is easily checked that
$(M_0,\germ m,\Phi,C_0,\germ c,\lambda)$
is a polarizing system of 
$(N_0,\germ n)$, and 
$(\sigma_{\mu},\rho^{\sigma_{\mu}},\mathcal K_{\mu})$
is consistent with
$(M_0,\germ m,\Phi,C_0,\germ c,\lambda)$.

{\it Case II : $(N_0,\germ n)$ is reduced and $\dim\mathcal Z(\germ n)>1$.}
Since the action of $\mathcal Z(\germ n)$ is via scalar 
multiplication \cite[Lemma 5]{vara},
it is easily seen that $(\pi,\rho^\pi,\mathcal H)$
factors through a representation of a quotient $(N'_0,\germ n')$
of $(N_0,\germ n)$ where the kernel of the quotient corresponds to a 
subalgebra of codimension one in $\mathcal Z(\germ n)$.
Again $\dim \germ n'<\dim\germ n$, 
and an argument similar to Case I above
applies.

{\it Case III : $(N_0,\germ n)$ is reduced and $\dim\mathcal Z(\germ n)=1$.}
In this case 
one of the statements of Proposition \ref{kirillovlemma} must hold. If statement
(b) of Proposition \ref{kirillovlemma}
holds, then by Proposition \ref{propcl} there is nothing left to prove. 
Next suppose that statement (a) of Proposition \ref{kirillovlemma}
holds. If the restriction of $\pi^\infty$ to $\mathcal Z(\germ n)$ is 
trivial, then an argument similar to Case II above applies.
If the restriction of $\pi^\infty$ to $\mathcal Z(\germ n)$ is 
not trivial, then from
Proposition \ref{induction-prop}
it follows that there exists an irreducible
unitary representation 
$(\sigma,\rho^\sigma,\mathcal K)$
of $(N'_0,\germ n')$
such that 
\begin{equation*}
(\pi,\rho^\pi,\mathcal H)=\ind_{(N'_0,\germ n')}^{(N_0,\germ n)}(\sigma,\rho^{\sigma},\mathcal K).
\end{equation*}
Here $(N'_0,\germ n')$ is the super Lie group 
identified by statement (a) of Proposition
\ref{kirillovlemma}.

By induction hypothesis, there exists a polarizing system 
$(M_0',\germ m',\Phi',C_0',\germ c',\lambda')$
of
$(N'_0,\germ n')$
and an irreducible unitary representation $(\sigma_\mu,\rho^{\sigma_\mu},\mathcal K_\mu)$
of $(C_0',\germ c')$
which is consistent with $(M_0',\germ m',\Phi',C_0',\germ c',\lambda')$ such that
$$
(\sigma,\rho^\sigma,\mathcal K)=
\ind_{(M'_0,\germ m')}^{(N'_0,\germ n')}
(\sigma_\mu\circ\Phi',\rho^{\sigma_\mu\circ\Phi'},\mathcal K_\mu).
$$
By Proposition \ref{indtransitivity}
we have 
\begin{equation}
\label{hopefullylast}
(\pi,\rho^\pi,\mathcal H)=
\ind_{(M'_0,\germ m')}^{(N_0,\germ n)}
(\sigma_\mu\circ\Phi',\rho^{\sigma_\mu\circ\Phi'},\mathcal K_\mu).
\end{equation}
Let $\tilde{\lambda}$ be an arbitrary $\mathbb R$-linear
extension of $\lambda'$ to $\germ n_0$, and set
$$
(M_0,\germ m,\Phi,C_0,\germ c,\lambda)=
(M_0',\germ m',\Phi',C_0',\germ c',\tilde{\lambda}).
$$
To show that $(M_0,\germ m,\Phi,C_0,\germ c,\lambda)$ is a polarizing system of 
$(N_0,\germ n)$, it suffices to check that $\germ m_0$ is a polarizing
subalgebra of $\germ n_0$ corresponding to ${\lambda}$. Let $X,Y,Z\in\germ n_0$ be 
chosen as in part (a) of Proposition \ref{kirillovlemma}.
From $Z\in\mathcal Z(\germ n'_0)$ and the fact that 
$\germ m_0$ is a polarizing
subalgebra of $\germ n'_0$ corresponding to $\lambda'$, 
it follows that $Z\in\germ m_0$. 
Hence for every $t\in\mathbb R$ and $v\in\mathcal K_\mu$,
$$
\Big(\sigma_\mu\circ\Phi\big(\exp(tZ)\big)\Big)v=e^{t\lambda(Z)\sqrt{-1}}v.
$$
Using (\ref{hopefullylast}) and the realization of induced representations 
given in Section \ref{realization} it is easy
to check that for every $t\in\mathbb R$ and $v\in\mathcal H$,
$$
\pi(\exp(tZ))v=e^{t\lambda(Z)\sqrt{-1}}v.
$$
Since the restriction of $\pi^\infty$ to $\mathcal Z(\germ n)$ is assumed to
be nontrivial, it follows that $\lambda(Z)\neq 0$.

Consider the skew-symmetric bilinear form
$$
\omega_\lambda:\germ n_0\times\germ n_0\to \mathbb R
$$
defined by  $\omega_\lambda(V,W)=\lambda([V,W])$, and
let 
$\omega'_\lambda$ be the restriction of ${\omega_\lambda}$
to $\germ n_0'\times \germ n_0'$.
Since $\germ m_0$ is a maximal isotropic subspace of
$\germ n_0'$, we have 
$$
\dim \germ m_0={1\over 2}(\dim\germ n_0'+\dim\germ s'_\lambda)
$$
where $\germ s'_\lambda$ is the radical of $\omega_\lambda'$.
To show that $\germ m_0$ is a maximal isotropic subspace of 
$\omega_\lambda$,
it suffices to prove that 
$$
\dim\germ s_\lambda=\dim\germ s'_\lambda-1$$
where
$\germ s_\lambda$ is the radical of $\omega_\lambda$.
Let $V\in\germ s_\lambda$, and write $V=aX+W$ where $a\in\mathbb R$ and 
$W\in\germ n_0'$. From $[Y,\germ n_0']=\{0\}$ it follows that
$$
\omega_\lambda(V,Y)=\lambda([V,Y])=a\lambda(Z)
$$
which implies that $a=0$, i.e. $V\in\germ n'_0$. Consequently, 
$\germ s_\lambda\subseteq \germ s'_\lambda$. Moreover, 
$[Y,\germ n_0']=\{0\}$ implies that $Y\in\germ s'_\lambda$, but 
$\lambda([X,Y])\neq 0$ implies that $Y\notin\germ s_\lambda$.
Thus $\dim\germ s_\lambda<\dim\germ s'_\lambda$, from which
it readily follows that $\dim\germ s_\lambda=\dim\germ s'_\lambda-1$.
Finally, verifying that 
$(\sigma_\mu,\rho^{\sigma_\mu},\mathcal K_\mu)$
is consistent with 
$(M_0,\germ m,\Phi,C_0,\germ c,\lambda)$
is trivial.

Next we prove part (b) of Theorem \ref{thmindpol}.
Suppose that $\chi:C_0\to\mathbb C^\times$ 
(respectively, $\chi':C_0'\to\mathbb C^\times$)
is the central character of $(\sigma_\mu,\rho^{\sigma_\mu},\mathcal K_\mu)$
(respectively, $(\sigma_{\mu'},\rho^{\sigma_{\mu'}},\mathcal K_{\mu'})$).
Since $\germ m_0$ is a polarizing subalgebra of 
$\germ n_0$ corresponding to $\lambda$ and
$(\sigma_\mu,\rho^{\sigma_\mu},\mathcal K_\mu)$
is consistent with $(M_0,\germ m,\Phi,C_0,\germ c,\lambda)$,
the representation 
$\ind_{M_0}^{N_0}\chi\circ\Phi$ is irreducible.
Since
$\pi=\ind_{M_0}^{N_0}\sigma_\mu\circ\Phi$, 
it follows that the unitary representation $(\pi,\mathcal H)$ of the 
nilpotent Lie group $N_0$ is a direct sum of 
$\dim \mathcal K_\mu$ copies of $\ind_{M_0}^{N_0}\chi\circ\Phi$.
With a similar argument, one can see that $(\pi,\mathcal H)$
is a direct sum of $\dim\mathcal K_{\mu'}$ copies of
the irreducible unitary representation 
$\ind_{M_0'}^{N_0}\chi'\circ\Phi'$. Consequently,
$\dim \mathcal K_\mu=\dim \mathcal K_{\mu'}$, which immediately implies 
that $(C_0,\germ c)$ and $(C_0',\germ c')$
are isomorphic. Moreover, we have
$$
\ind_{M_0}^{N_0}\chi\circ\Phi\simeq\ind_{M_0'}^{N_0}\chi'\circ\Phi'
$$
and Kirillov theory for nilpotent Lie
groups (e.g., \cite[Theorem 2.2.4]{corgr}) implies that 
$$
\lambda'=\mathrm{Ad}^*(n)(\lambda)
$$ 
for some $n\in N_0$.

\end{proof}
\begin{corollary}
\label{mycorollary}
Let $(\pi,\rho^\pi,\mathcal H)$ be an irreducible
unitary representation of a nilpotent
Lie supergroup $(N_0,\germ n)$, and $(\pi,\mathcal H)$ be the 
unitary representation of $N_0$ obtained as restriction of 
$(\pi,\rho^\pi,\mathcal H)$ to the even part. Then
there exists an irreducible unitary representation 
$(\sigma,\mathcal K)$ of $N_0$ 
such that $(\pi,\mathcal H)$ is a direct sum of $2^l$ 
copies of $(\sigma,\mathcal K)$, where $l$ is a 
nonnegative integer.
\end{corollary}
\begin{proof}
Since special induction commutes with restriction
to the even part, this
follows immediately from part (a) of 
Theorem \ref{thmindpol} and the fact that $\dim\mathcal K_\mu=2^l$
for some $l\geq 0$.
\end{proof}

\noindent{\bf Remark.} Suppose that an irreducible unitary representation
$(\pi,\rho^\pi,\mathcal H)$ is given by (\ref{whatisit}).
If we set $\kappa(\pi,\rho^\pi,\mathcal H)=\dim \germ c$, then by 
part (b) of Theorem \ref{thmindpol} the positive 
integer $\kappa(\pi,\rho^\pi,\mathcal H)$ does not depend on the 
choice of the polarizing system and hence is an invariant of 
$(\pi,\rho^\pi,\mathcal H)$. 
In fact using Corollary \ref{thmindpol}
one can see that 
$\kappa(\pi,\rho^\pi,\mathcal H)$ can be obtained as follows. 
Consider the representation
$(\pi,\mathcal H)$ of the Lie group $N_0$
obtained by restriction of $(\pi,\rho^\pi,\mathcal H)$ to the 
even part of $(N_0,\germ n)$. 
The representation $(\pi,\mathcal H)$ 
is always a direct sum of $2^r$ copies of an irreducible 
unitary representation $(\pi',\mathcal H')$ of $N_0$, where 
$r$ is a nonnegative integer. In the latter case, we have 
$$
\kappa(\pi,\rho^\pi,\mathcal H)=\left\{
\begin{array}{ll}
2r\quad &\textrm{if }(\pi,\rho^\pi,\mathcal H)\simeq
(\pi,\rho^\pi,\mathrm\Pi\mathcal H),\\
2r+1\quad&\textrm{otherwise.}

\end{array}
\right.
$$
In particular, when $r=0$ the representation 
$(\pi,\rho^\pi,\mathcal H)$ is purely even and therefore 
$\kappa(\pi,\rho^\pi,\mathcal H)=1$.\\

\subsection{Irreducibility of codimension-one induction}
In this section we prove 
that induction from a polarizing system always
yields an irreducible unitary representation.

\begin{theorem}
\label{lastmainthm}
Let $(M_0,\germ m,\Phi,C_0,\germ c,\lambda)$ be a polarizing system of 
$(N_0,\germ n)$.
Suppose that
$(\sigma_\mu,\rho^{\sigma_\mu},\mathcal K_\mu)$ is 
the representation of 
$(C_0,\germ c)$ consistent with this polarizing system.
Then the unitary representation 
$$
(\pi,\rho^\pi,\mathcal H)=\ind_{(M_0,\germ m)}^{(N_0,\germ n)}
(\sigma_\mu,\rho^{\sigma_\mu},\mathcal K_\mu)
$$
is irreducible.
\end{theorem}
\begin{proof}
We prove the theorem by induction on $\dim \germ n$. 
If $\lambda=0$, then $\germ m=\germ n$ and 
$\mathrm{ker}\,\Phi\supseteq\germ n_0$, which implies that
$\germ c=\{0\}$ and therefore $(\pi,\rho^\pi,\mathcal H)$
is the trivial representation. Without loss of generality,
from now on 
we assume that $\lambda\neq 0$.
There are
three cases to consider.

{\it Case I : $(N_0,\germ n)$ is not reduced.} 
Recall that $\germ a[\germ n]$ is a $\mathbb Z_2$-graded ideal of 
$\germ n$. Since $\germ a[\germ n]\neq\{0\}$, we have 
$\mathcal Z(\germ n)\cap \germ a[\germ n]\neq\{0\}$.
Indeed let $\germ e^{(0)}=\germ a[\germ n]$ and for any 
positive integer $j$, set $\germ e^{(j+1)}=[\germ n,\germ e^{(j)}]$.
Let $j_0=\min\{\ j\ |\ \germ e^{(j)}=\{0\}\ \}$. Then 
$\germ e^{(j_0-1)}\subseteq\germ a[\germ n]\cap\mathcal Z(\germ n)$.

Let $W\in\mathcal Z(\germ n)\cap\germ a[\germ n]$. 
Since $\mathcal Z(\germ n)\cap\germ a[\germ n]$ is $\mathbb Z_2$-graded,
we can choose $W$ suitably such that $W\in\germ n_0$ or $W\in\germ n_1$.
If $W\in\germ n_1$ then obviously $W\in\germ m_1$, and if $W\in\germ n_0$, then 
$W\in\germ m_0$ because otherwise 
$\germ m'_0=\germ m_0+\mathbb RW$ is a Lie subalgebra of 
$\germ n_0$ with the property that 
$\lambda([\germ m_0',\germ m_0'])=\{0\}$, and 
the latter implies that 
$(M_0,\germ m,\Phi,C_0,\germ c,\lambda)$ does not satisfy part (b) of 
Definition \ref{polarizing}. Therefore we have
shown that $\mathcal Z(\germ n)\cap\germ a[\germ n]\subseteq\germ m$.

Our next task is to show that $\Phi(W)=0$ for every 
$W\in\mathcal Z(\germ n)\cap\germ a[\germ n]$. Without
loss of generality, we can assume $W\in\germ n_0$ or $W\in\germ n_1$.
If $W\in\germ n_1$, then
we have $[W,W]=0$ which implies that $[\Phi(W),\Phi(W)]=0$. Since
$(C_0,\germ c)$ is reduced, we have $\Phi(W)=0$. If $W\in\germ n_0$
then for every $v\in\mathcal K_\mu$ we have
\begin{equation*}
\big(\sigma_\mu^\infty\circ\Phi(W)\big)v=\mu\circ\Phi(W)\sqrt{-1}\,v.
\end{equation*}
Since $\lambda\neq 0$, by Lemma \ref{nonzeropolar}
and surjectivity of $\Phi$ it follows 
that $\mu\neq 0$.
From the 
realization of induced representations 
given in Section 
\ref{realization} and the fact that 
$W\in\mathcal Z(\germ n)$
it is easily seen that for every $v\in\mathcal H$,
$$
\pi^\infty(W)v=\lambda(W)\sqrt{-1}\,v=\mu\circ\Phi(W)\sqrt{-1}\,v.
$$
If $\Phi(W)\neq 0$,  
then 
$\mu\circ\Phi(W)\neq 0$ from which it follows that
$\pi^\infty(W)\neq 0$, which contradicts the fact that 
by Proposition \ref{nilrad} we have  
$\pi^\infty(W)=0$.

Set $\germ s=\mathcal Z(\germ n)\cap\germ a[\germ n]$ and consider the super 
Lie group $(N'_0,\germ n')$ where $\germ n'=\germ n/\germ s$.
(Thus $N_0'=N_0/S_0$ where $S_0=\{\ \exp(tV)\ |\ V\in\germ s_0\ \}$.)
Obviously $(\pi,\rho^\pi,\mathcal H)$ factors through $(N_0',\germ n')$.
We denote this representation of $(N_0',\germ n')$ by 
$(\overline\pi,\rho^{\overline\pi},\mathcal H)$.
Moreover, 
$$
\mathcal Z(\germ n)\cap \germ a[\germ n]\subseteq\mathcal Z(\germ n)\subseteq\germ m.
$$
Since $\ker\Phi\cap\germ m_0=\ker\lambda\cap\germ m_0$ we have
$\mathcal Z(\germ n)\cap\germ a[\germ n]\cap\germ n_0\subseteq\ker\lambda$.
Therefore the polarizing system $(M_0,\germ m,\Phi,C_0,\germ c,\lambda)$ corresponds 
via the quotient map $\mathsf q:\germ n\to\germ n'$
to a polarizing system $(M'_0,\germ m',\Phi',C_0,\germ c,\lambda')$ of 
$(N'_0,\germ n')$.
Moreover, $(\sigma_\mu,\rho^{\sigma_\mu},\mathcal K_\mu)$
is consistent with $(M'_0,\germ m',\Phi',C_0,\germ c,\lambda')$.
We can express
$(\overline\pi,\rho^{\overline\pi},\mathcal H)$ as 
$$
(\overline\pi,\rho^{\overline\pi},\mathcal H)=\ind_{(M_0',\germ m')}^{(N_0',\germ n')}
(\sigma_\mu\circ\Phi',\rho^{\sigma_\mu}\circ\Phi',\mathcal K_\mu).
$$
Since $\dim\germ n'<\dim\germ n$, by the induction hypothesis it follows that
$(\overline\pi,\rho^{\overline\pi},\mathcal H)$
(and hence 
$(\pi,\rho^\pi,\mathcal H)$) is irreducible.

{\it Case II : $(N_0,\germ n)$ is reduced 
and $\mathcal Z(\germ n)\cap\ker\lambda\neq\{0\}$.} 
In this case $\mathcal Z(\germ n)\cap\ker\lambda$ is an ideal of 
$\germ n$, and the fact that $\germ m_0$ is a polarizing subalgebra of 
$\germ n_0$ corresponding to $\lambda$ implies that 
$\mathcal Z(\germ n)\subseteq\germ m_0$.
The representation $(\pi,\rho^\pi,\mathcal H)$ factors 
through $(N'_0,\germ n')$ where 
$$
\germ n'=\germ n/\mathcal Z(\germ n)\cap\ker\lambda,
$$
and the polarizing system 
$(M_0,\germ m,\Phi,C_0,\germ c,\lambda)$ corresponds 
via the quotient map $\mathsf q:\germ n\to\germ n'$
to a polarizing system $(M'_0,\germ m',\Phi',C_0,\germ c,\lambda')$ of 
$(N'_0,\germ n')$.
The rest of the argument is similar to 
Case I. 

{\it Case III : $(N_0,\germ n)$ is reduced and 
$\mathcal Z(\germ n)\cap\ker\lambda=\{0\}$.} 
It follows that $\dim\mathcal Z(\germ n)=1$, hence
one of the statements of Proposition \ref{kirillovlemma} should hold.
If statement (b) of Proposition  \ref{kirillovlemma} holds, then
there is essentially nothing left to prove. From now on we assume that
statement (a) of Proposition  \ref{kirillovlemma} holds. Let $X,Y,Z$, 
$\germ n'$, and  
$\germ w$ be as in part (a) of 
Proposition  \ref{kirillovlemma}. 

Our first task is to show that without loss of generality,
we can assume that $\lambda(Y)=0$ and $\lambda(Z)\neq 0$. Indeed
one can modify the choice of the polarizing system 
as follows. Since $\germ m_0$ is a 
polarizing Lie subalgebra of $\germ n_0$ corresponding to $\lambda$,
we should have $Z\in\germ m_0$, and since 
$\mathcal Z(\germ n)\cap\ker\lambda=\{0\}$, 
we should have 
$\lambda(Z)\neq 0$.
For every $n\in N_0$,
we have a polarizing system
$$
(\,nM_0n^{-1},\mathrm{Ad}(n)(\germ m),\Phi\circ\mathrm{Ad}(n^{-1}),
C_0,\germ c,
\mathrm{Ad}^*(n)(\lambda)\,)
$$ 
in $(N_0,\germ n)$
and one can see that  
$$
(\pi,\rho^\pi,\mathcal H)\simeq\ind_{(nM_0n^{-1},\mathrm{Ad}(n)(\germ m))}^{(N_0,\germ n)}
(\sigma_\mu\circ\Phi\circ\mathrm{Ad}(n^{-1}),\rho^{\sigma_\mu\circ\Phi\circ\mathrm{Ad}(n^{-1})},\mathcal K_\mu).
$$
In particular, if we set 
$n=\exp(t_\circ X)$ where $t_\circ={\lambda(Y)\over \lambda(Z)}$, then
$$
\big(\mathrm{Ad}^*(n)(\lambda)\big)(Y)=\lambda(Y-{\lambda(Y)\over\lambda(Z)}Z)=0.
$$
The condition 
$\mathcal Z(\germ n)\cap \ker\big(\mathrm{Ad}^*(n)(\lambda)\big)=\{0\}$ is easy to check
as well.

From now on we assume that $\lambda(Y)=0$ and $\lambda(Z)\neq 0$.
Our next task is to prove that without loss 
of generality we can also assume that
$\germ m\subseteq \germ n'$. Suppose, on the contrary, that 
$\germ m\nsubseteq \germ n'$. In this case we show that 
$(\pi,\rho^\pi,\mathcal H)$ 
is unitarily equivalent to a representation
$$
(\pi',\rho^{\pi'},\mathcal H')=\ind_{(M_0',\germ m')}^{(N_0,\germ n)}
(\sigma_{\mu}\circ\Phi',\rho^{\sigma_{\mu}\circ\Phi'},\mathcal K_{\mu})
$$
where $(\sigma_{\mu},\rho^{\sigma_{\mu}},\mathcal K_{\mu})$ is
consistent with a polarizing system 
$(M_0',\germ m',\Phi',C_0,\germ c,\lambda)$ which satisfies 
$\germ m'\subseteq\germ n'$.
To this end, first note that in part (a) 
of Proposition  \ref{kirillovlemma}, we can choose $X$  
such that 
$\lambda(X)=0$ and 
$$
\germ m=\mathbb R X\oplus\mathbb RZ\oplus\germ w'_0\oplus\germ n_1
$$
where $\germ w_0'$ is a subspace of $\germ n_0'$ such that $\lambda(\germ w_0')=0$.
Indeed since $\germ m\nsubseteq \germ n'$, 
we can choose $X$ such that $X\in\germ m_0$. If $\lambda(X)\neq 0$, then since
$Z\in\germ m$ and 
$\lambda(Z)\neq 0$
we can substitute $X$ by $X-{\lambda(X)\over\lambda(Z)}Z$.
In a similar fashion we can choose a
complement $\germ w_0'$ to $\mathbb RZ$ in $\germ m\cap\germ n'$ which is included
in $\ker\lambda$.
Next note that $Y\notin\germ m_0$ because otherwise
$\lambda([X,Y])=\lambda(Z)\neq 0$ which contradicts the fact
that $\germ m_0$ is a polarizing subalgebra of $\germ n_0$ 
corresponding to $\lambda$. 
 Consider the subalgebra $\germ m'$ of $\germ n'$ defined
by
$$
\germ m'=\mathbb RY\oplus\mathbb RZ\oplus\germ w_0'\oplus\germ n_1.
$$
To show that $\germ m'$ is a subalgebra of $\germ n'$, note that
\begin{equation}
\label{msefrmsefr}
[\germ m'_0,\germ m'_0]\subseteq[\germ w_0',\germ w_0']
\subseteq\germ m_0\cap\germ n_0'\subsetneq\germ m'_0
\end{equation}
and
$$
[\germ n_1,\germ n_1]\subseteq\germ n_0'\cap\germ m_0\subsetneq\germ m_0'.
$$
Let $M'_0$ be the Lie subgroup of $N_0$ corresponding
to $\germ m'_0$. We define 
$$
\Phi':(M_0',\germ m')\to (C_0,\germ c)
$$
as follows.
For every $W\in\mathbb RZ\oplus\germ w_0'\oplus\germ n_1$
and $W'\in\mathbb RY$ we set
\begin{equation*}
\Phi'(W+W')=\Phi(W).
\end{equation*}
We now prove that $(M_0',\germ m',\Phi',C_0,\germ c,\lambda)$ is a polarizing system and
$(\sigma_{\mu},\rho^{\sigma_{\mu}},\mathcal K_{\mu})$ is consistent with
it. 
From a calculation similar to (\ref{msefrmsefr}) it follows that
$$
\lambda([\germ m_0',\germ m_0'])\subseteq\lambda([\germ m_0,\germ m_0])=\{0\}.
$$
Moreover, $Y\notin\germ m_0$ because otherwise we have $Z\in[\germ m_0,\germ m_0]$
and $\lambda(Z)\neq 0$ which contradicts the fact that $\germ m_0$ is a polarizing subalgebra of $\germ n_0$ corresponding to $\lambda$. 
Therefore we have $\dim\germ m_0'=\dim\germ m_0$, which
implies that $\germ m_0'$ is a polarizing subalgebra of $\germ n_0$. 
Using $[Y,\germ n']=\{0\}$ it is easy to check that
part (c) of Definition \ref{polarizing} holds. Part (d) 
of Definition \ref{polarizing} follows from 
$\lambda(X)=\lambda(Y)=0$ and $\lambda(Z)\neq 0$.
Finally, one can check that 
$(\sigma_{\mu},\rho^{\sigma_{\mu}},\mathcal K_{\mu})$
is consistent with $(M_0',\germ m',\Phi',C_0,\germ c,\lambda)$.

To prove that $(\pi,\rho^\pi,\mathcal H)\simeq(\pi',\rho^{\pi'},\mathcal H')$,
it suffices to show that 
\begin{equation}
\label{twoindeql}
\ind_{(M_0,\germ m)}^{(M_0'',\germ m'')}
(\sigma_{\mu}\circ\Phi,\rho^{\sigma_{\mu}\circ\Phi},\mathcal K_{\mu})
\simeq
\ind_{(M_0',\germ m')}^{(M_0'',\germ m'')}
(\sigma_{\mu}\circ\Phi',\rho^{\sigma_{\mu}\circ\Phi'},\mathcal K_{\mu})
\end{equation}
where $M_0''=M_0M_0'$ and 
$\germ m''=\germ m+\germ m'$, i.e.,
$$
\germ m''=\mathbb RX\oplus\mathbb RY\oplus\mathbb RZ\oplus\germ w_0'\oplus\germ n_1.
$$
Since $\germ m''$ is $\mathbb Z_2$-graded, we can express it as
$\germ m''=\germ m''_0\oplus\germ m''_1$.
Observe that the vector space 
$\germ w_0'$ is in fact an ideal of $\germ m_0''$. 
To prove the latter statement, note that since $\germ m_0$ and $\germ m_0'$
are polarizing subalgebras of $\germ n_0$ corresponding to $\lambda$, we 
should have
$$
[\germ m_0,\germ m_0]\subseteq\mathbb RX\oplus \germ w_0'
\textrm{\ \ and\ \ }
[\germ m_0',\germ m_0']\subseteq\mathbb RY\oplus \germ w_0'
$$
which imply that $\mathbb RX\oplus \germ w_0'$ and 
$\mathbb RY\oplus \germ w_0'$ are Lie subalgebras of $\germ n_0$.
Since 
$$
\germ w_0'=(\mathbb RX\oplus\germ w_0')\cap(\mathbb RY\oplus\germ w_0'),
$$
the vector space $\germ w_0'$
is in fact a Lie subalgebra of both of $\mathbb RX\oplus \germ w_0'$ and 
$\mathbb RY\oplus \germ w_0'$. But in a nilpotent Lie algebra, any Lie 
subalgebra
of codimension one is an ideal. Therefore $\germ w_0'$ is an ideal in
both 
$\mathbb RX\oplus \germ w_0'$ and 
$\mathbb RY\oplus \germ w_0'$. It follows that $\germ w_0'$ is an ideal
in the Lie algebra generated by $\mathbb RX\oplus \germ w_0'$ and 
$\mathbb RY\oplus \germ w_0'$, i.e., in 
$\germ m_0''=\mathbb RX\oplus\mathbb RY\oplus\mathbb RZ\oplus\germ w_0'$.

Next we obtain the unitary equivalence of (\ref{twoindeql}). 
Let $E_0=\{\ \exp(tZ)\ |\ t\in\mathbb R\ \}$ and $\chi:E_0\to\mathbb C^\times$
be the unitary character given by 
$$
\chi(\exp(tZ))=e^{t\lambda(Z)\sqrt{-1}}.
$$
If $(\pi_L,\rho^{\pi_L},\mathcal H_L)$ 
denotes the representation on the left hand side of
(\ref{twoindeql}), then 
we can realize $\mathcal H_L$ as $L^2(\mathbb R,\mathcal K_\mu)$
such that the action of $(M_0'',\germ m'')$ is given as follows. 
For every $y,t\in\mathbb R$, $W'\in\germ w_0'$, and $f\in L^2(\mathbb R,\mathcal K_\mu)$,
we have
\begin{eqnarray*}
\big(\pi_L(\exp(tX))f\big)(y)&=&\chi(ty)f(y)\\
\big(\pi_L(\exp(tY))f\big)(y)&=&f(y+t)\\
\big(\pi_L(\exp(tZ))f\big)(y)&=&\chi(t)f(y)\\
\big(\pi_L(\exp(W')f\big)(y)&=&f(y).
\end{eqnarray*}
Moreover, if $f\in L^2(\mathbb R,\mathcal K_\mu)$ is in the Schwartz space 
then from $[Y,\germ n_1]=\{0\}$ it follows that for every 
$W\in\germ n_1$ and $y\in\mathbb R$ we have
$$
(\rho^{\pi_L}(W))(y)=\Phi\big(\exp(yY)\cdot W\big)\big(f(y)\big)=
\Phi(W)\big(f(y)\big).
$$
Similarly, if $(\pi_R,\rho^{\pi_R},\mathcal H_R)$ 
denotes the representation on the right hand side of 
(\ref{twoindeql}), then $(\pi_R,\rho^{\pi_R},\mathcal H_R)$ can also
be realized on $L^2(\mathbb R,\mathcal K_\mu)$ as follows.
For every $x,t\in\mathbb R$, $W'\in\germ w_0'$, and $f\in L^2(\mathbb R,\mathcal K_\mu)$,
we have
\begin{eqnarray*}
\big(\pi_R(\exp(tX))f\big)(x)&=&f(x+t)\\
\big(\pi_R(\exp(tY))f\big)(x)&=&\chi(-tx)f(x)\\
\big(\pi_R(\exp(tZ))f\big)(x)&=&\chi(t)f(x)\\
\big(\pi_R(\exp(W')f\big)(x)&=&f(x).
\end{eqnarray*}
Moreover, if
$f\in L^2(\mathbb R,\mathcal K_\mu)$
is indeed in the Schwartz space, then for every $W\in\germ n_1$ and 
$x\in\mathbb R$ we have
$$
(\rho^{\pi_R}(W))(x)=\Phi'\big(\exp(xX)\cdot W\big)
\big(f(x)\big)=\Phi(W)\big(f(x)\big)
$$
where the last equality follows from the fact that $\Phi(X)=0$ and thus
\begin{eqnarray*}
\Phi'(\exp(xX)\cdot W)&=&\Phi'(W+[X,W]+{1\over 2}[X,[X,W]]+\cdots)\\
&=&\Phi(W+[X,W]+{1\over 2}[X,[X,W]]+\cdots)\\
&=&\Phi(W)+[\Phi(X),\Phi(W)]+{1\over 2}[\Phi(X),[\Phi(X),\Phi(W)]]+\cdots\\
&=&\Phi(W).
\end{eqnarray*}
It is now easy to check that the 
isometry $T:\mathcal H_L\to\mathcal H_R$
which intertwines $(\pi_L,\rho^{\pi_L},\mathcal H_L)$
and $(\pi_R,\rho^{\pi_R},\mathcal H_R)$
is given by the Fourier transform, i.e., 
$$
Tf(x)=\int_{-\infty}^\infty\chi(xy)f(y)dy.
$$

We now complete the proof of Case III.
The proof closely follows an argument that is given in \cite[p. 63]{corgr}.
Recall that as shown above, we can assume that $\germ m\subseteq\germ n'$.
It follows that $(M_0,\germ m,\Phi,C_0,\germ c,\lambda)$ is a polarizing system 
in $(N_0',\germ n')$. Since $\dim\germ n'<\dim\germ n$, by induction hypothesis
the representation 
\begin{equation}
\label{zegond}
(\pi'',\rho^{\pi''},\mathcal H'')=\ind_{(M_0,\germ m)}^{(N_0',\germ n')}
(\sigma_\mu\circ\Phi,\rho^{\sigma_\mu\circ\Phi},\mathcal K_\mu)
\end{equation}
is irreducible. Since $Z\in\mathcal Z(\germ n')$, 
by \cite[Lemma 5]{vara} there exists a real number $b\in\mathbb R$ such that
for every $t\in\mathbb R$ and $v\in\mathcal H''$ we have
$$
\pi''(\exp(tZ))v=e^{tb\sqrt{-1}}v.
$$
Recall that $\lambda(Z)\neq 0$ and $\lambda(Y)=0$.
Since
$Z\in\mathcal Z(\germ n)\cap \germ n_0'\subseteq\germ m_0$,
from (\ref{zegond}) and the realization
of the induced representation (see Section \ref{realization}) it
follows that 
$$
\pi''(\exp(tZ))v=e^{t\lambda(Z)\sqrt{-1}}v
$$ 
and therefore 
$b\neq 0$.
Next observe that
$$
(\pi,\rho^\pi,\mathcal H)=\ind_{(N_0',\germ n')}^{(N_0,\germ n)}
(\pi'',\rho^{\pi''},\mathcal H'')
$$
and by Section \ref{realization} 
we can assume $\mathcal H=L^2(\mathbb R,\mathcal H'')$ where 
for every $f\in L^2(\mathbb R,\mathcal H'')$, 
$s\in\mathbb R$, $t\in\mathbb R$,
and $n\in N_0'$ we have
\begin{eqnarray}
\label{txaction}
\big(\pi(\exp(tX))f\big)(s)&=&f(s+t)
\end{eqnarray}
and
\begin{eqnarray*}
(\pi(n)f)(s)&=&\pi''\big(\exp(sX)n\exp(-sX)\big)\big(f(s)\big).
\end{eqnarray*}
In particular, since $\spn_\mathbb R\{X,Y,Z\}$ is a Heisenberg
Lie algebra, we have
$$
\big(\pi(\exp(tY))f\big)(s)=e^{stb\sqrt{-1}}f(s).
$$
Moreover, if $f\in L^2(\mathbb R,\mathcal H'')$ is a smooth vector for the 
action of $\pi$ and has compact support, then 
for every $W\in\germ n_1$ and $s\in\mathbb R$ we have 
$$
\big(\rho^\pi(W)f\big)(s)=\rho^{\pi''}(\exp(sX)\cdot W)\big(f(s)\big).
$$
Let $T:L^2(\mathbb R,\mathcal H'')\to L^2(\mathbb R,\mathcal H'')$
be a bounded even linear operator which intertwines $(\pi,\rho^\pi,\mathcal H)$
with itself.
To complete the proof of Case III, it suffices to show that $T$ is a scalar multiple
of the identity. 
From \cite[Lemma 2.3.3]{corgr}, \cite[Lemma 2.3.2]{corgr} and 
\cite[Lemma 2.3.1]{corgr} it follows that there exists a family
$\{T_t\}_{t\in\mathbb R}$ of even linear operators $T_t:\mathcal H''\to\mathcal H''$
such that $||T_t||\leq||T||$ for every $t\in\mathbb R$, and
for every $f\in L^2(\mathbb R,\mathcal H'')$ we have 
$Tf\,(t)=T_t(f(t))$. One can check that $T_t$ intertwines the action
of the representation $(\pi''_t,\rho^{\pi''_t},\mathcal H'')$
of $(N'_0,\germ n')$ which is defined
by 
$$
\pi''_t(n)=\pi''\big(\exp(tX)n\exp(-tX)\big)\textrm{\ \ \ and\ \ \ }
\rho^{\pi''_t}(W)=
\rho^{\pi''}\big(\exp(tX)\cdot W\big).
$$
But $(\pi''_t,\rho^{\pi''_t},\mathcal H'')$ is irreducible, 
and from \cite[Lemma 5]{corgr} it follows that for every $t\in\mathbb R$, the operator
$T_t$ is multiplication by a scalar $\gamma(t)$. From (\ref{txaction}) it follows that $\gamma(t)$ does not depend on $t$,
i.e., $T$ is a scalar multiple of identity.

\end{proof}


\subsection{Existence of suitable polarizing subalgebras}
\label{suitable}

In this section 
$\germ n=\germ n_0\oplus\germ n_1$ will be a nilpotent Lie superalgebra. 
In this technical section we prove the existence of a special kind of
polarizing subalgebras in $\germ n_0$. The main goal of our 
fairly complicated arguments is to prove 
Lemma \ref{existenceofevenpart}.

For every $\lambda\in\germ n_0^*$ we consider the 
symmetric bilinear form 
$$
\mathrm B_\lambda:\germ n_1\times \germ n_1\to\mathbb R
$$
defined by $\mathrm B_\lambda(X,Y)=\lambda([X,Y])$. 
We denote the radical of $\mathrm B_\lambda$ by $\germ r_\lambda$.

\begin{lemma}
\label{radicalelements}
Suppose $\lambda\in\germ n_0^*$ 
and $\mathrm B_\lambda$  
is nonnegative definite. If $X\in\germ n_1$ is an
isotropic vector, i.e., it satisfies 
$\mathrm B_\lambda(X,X)=0$, then $X\in\germ r_\lambda$.
\end{lemma}
\begin{proof}
Suppose, on the contrary, that there exists an element $Y\in\germ n_1$
such that 
$\lambda([X,Y])\neq 0$. Then 
for every $s\in\mathbb R$ we have
$$
\lambda([X+sY,X+sY])=\lambda([X,X])+2s\lambda([X,Y])+s^2\lambda([Y,Y]).
$$
Since $\lambda([X,X])=0$, one can find an $s\in\mathbb R$ such that 
$\lambda([X+sY,X+sY])<0$, which contradicts the fact that 
$\mathrm B_\lambda$ is nonnegative definite.

\end{proof}

In the rest of this section we fix $\lambda\in\germ n_0^*$ such that $\mathrm B_\lambda$
is nonnegative definite.
Suppose that $\germ n$ has a subalgebra 
$\germ n'=\germ n'_0\oplus\germ n'_1$ where $\germ n_0'=\germ n_0$ and  
$\dim \germ n_1'=\dim\germ n_1-1$.
Then it is easily checked that $\germ n'$ is indeed an ideal 
of $\germ n$ and $[\germ n,\germ n]\subset\germ n'$.
As a vector space, we can write $\germ n_1$ as a direct sum 
\begin{equation}
\label{directsm}
\germ n_1=\germ n_1'\oplus\mathbb RA
\end{equation}
for some $A\in\germ n_1$, and without loss of generality we can 
assume that $A$ is chosen suitably such that
\begin{equation}
\label{orthogg}
\mathrm B_\lambda(A,\germ n_1')=0.
\end{equation}
In the rest of this section we fix such an $A\in\germ n_1$.

\begin{lemma}
\label{complicatedlemma}
Let $E,F\in\germ n_1$. 
Suppose that 
\begin{itemize}
\item[(a)] $\mathrm B_\lambda([[A,E],[[A,E],F]]\,,\,[[A,E],[[A,E],F]])=0$,
\item[(b)] For every $G\in\germ n_1$, we have 
$\mathrm B_\lambda(X_G,X_G)=0$ where
$$
X_G=[[A,[A,[F,[F,[A,E]]]]],G].
$$
\end{itemize}
Then 
$\mathrm B_\lambda([[A,E],F],[[A,E],F])=0$.
\end{lemma}
\begin{proof}
Set $Y=[A,E]$. Our goal is to prove that
$$
\mathrm B_\lambda([F,Y],[F,Y])=0.
$$
Observe that by the Jacobi identity we have 
\begin{equation}
\label{first}
\lambda([[F,Y],[F,Y]])-\lambda([F,[Y,[F,Y]])+\lambda([Y,[[F,Y],F]])=0.
\end{equation}
Set $P=[[F,Y],F]$. Then 
$$
\lambda([Y,[[F,Y],F]])=
\lambda([[A,E],P])=
-\lambda([P,[A,E]])
$$ 
and by the Jacobi identity we have  
\begin{equation}
\label{second}
\lambda([P,[A,E]])+\lambda([A,[E,P]])-\lambda([E,[P,A]])=0.
\end{equation}
Since $[E,P]\in\germ n_1'$, 
by (\ref{orthogg})
we have $\lambda([A,[E,P]])=0$. To complete the proof of the lemma it suffices to prove that 
$\lambda([F,[Y,[F,Y]])=0$ and 
$\lambda([E,[P,A]])=0$.
By Lemma \ref{radicalelements}
it suffices to show that 
$$
[Y,[Y,F]]\textrm{\ \  and\ \  }[A,P]=[A,[F,[F,Y]]]
$$ 
are 
isotropic vectors for $\mathrm B_\lambda$. 
For $[Y,[Y,F]]$, the latter statement 
is assumption (a) of the lemma.
Next we prove that 
$$
\mathrm B_\lambda([A,P],[A,P])=0.
$$
By the Jacobi identity
we have
\begin{equation}
\label{third}
\lambda([[A,P],[A,P]])-\lambda([A,[P,[A,P]]])+
\lambda([P,[[A,P],A]])=0.
\end{equation}
Since 
$[P,[A,P]]\in\germ n_1'$, 
by (\ref{orthogg})
we have 
\begin{equation}
\label{threedemi}
\lambda([A,[P,[A,P]]])=0.
\end{equation}
By the Jacobi identity we have
\begin{eqnarray}
\notag
\lambda([[A,[A,P]],[F,[F,Y]]])&+&\lambda([F,[[F,Y],[A,[A,P]]]])\\
&-&\lambda([[F,Y],[[A,[A,P]],F]])=0.
\label{fourth}
\end{eqnarray}
Next observe that
$$
\lambda([F,[[F,Y],[A,[A,P]]]])=
-\lambda([F,[[A,[A,P]],[F,Y]]])
$$
and
$
[A,[A,P]]=[A,[A,[F,[F,[A,E]]]]]
$. Therefore from
assumption (b) of Lemma
\ref{complicatedlemma}
and Lemma \ref{radicalelements} it follows that
$$
\lambda([F,[[A,[A,P]],[F,Y]]])=0.
$$
A similar argument proves that
$$
\lambda([[F,Y],[[A,[A,P]],F]])=0.
$$
The last two equalities, together with
(\ref{fourth}), imply that 
\begin{equation}
\label{fiffth}
\lambda([[A,[A,P]],[F,[F,Y]]])=0.
\end{equation}
But
$$
[[A,[A,P]],[F,[F,Y]]]=[[A,[A,P]],P]=-[P,[A,[A,P]]]
$$
and therefore from (\ref{third}), (\ref{threedemi}), and (\ref{fiffth})
it follows that 
$$
\mathrm B_\lambda([A,P],[A,P])=0
$$ which
completes the proof.

\end{proof}
\begin{lemma}
\label{aef}
Let $E,F\in\germ n_1$. 
Then 
$\mathrm B_\lambda([[A,E],F],[[A,E],F]])=0$.
\end{lemma}
\begin{proof}
Set $\germ n^{(0)}=\germ n$ and 
$\germ n^{(i)}=[\germ n^{(0)},\germ n^{(i-1)}]$. Note that $\germ n^{(r)}=\{0\}$ 
for $r\gg 0$. 
We prove the lemma by a backward induction as follow.
We assume that the lemma holds for every $E,F$ such that 
$[[A,E],F]\in\germ n^{(r)}$, and we prove that
it holds for every $E,F$ such that $[[A,E],F]\in\germ n^{(r-1)}$. 

Assume the induction hypothesis, and consider $E,F\in\germ n_1$ such that 
$$
[[A,E],F]\in\germ n^{(r-1)}.
$$ It is easily seen that $[[A,E],[[A,E],F]]$
and every element of the form 
$$
[[A,[A,[F,[F,[A,E]]]]],G]
$$ 
where $G\in\germ n_1$
satisfy the induction hypothesis and therefore they are
isotropic vectors for $\mathrm B_\lambda$.
Lemma \ref{complicatedlemma} 
implies that $[[A,E],F]$ is an isotropic vector for 
$\mathrm B_\lambda$ as well.

\end{proof}

\begin{lemma}
\label{firststrangelemma}
Let $E,F\in\germ n_1$.
Then 
$\mathrm B_\lambda([A,[E,F]],[A,[E,F]])=0$.
\end{lemma}
\begin{proof}
Set $X=[E,F]$.
By the Jacobi identity we have
$$
\lambda([[A,X],[A,X]])-\lambda([A,[X,[A,X]]])+
\lambda([X,[[A,X],A])=0.
$$
Since 
$[X,[A,X]]\in\germ n_1'$, 
we have
$\lambda([A,[X,[A,X]]])=0$ and therefore 
\begin{equation}
\label{eqlity}
\lambda([[A,X],[A,X]])=
-\lambda([X,[[A,X],A]])= 
-\lambda([X,[A,[A,X]]]).
\end{equation}
To complete the proof of the lemma 
it suffices to show that the rightmost term in 
(\ref{eqlity}) vanishes.
By the Jacobi identity we have
\begin{eqnarray}
\label{anothjaq}
\lambda([[A,[A,X]],[E,F]])&+&\lambda([E,[F,[A,[A,X]]]])\\
&-&\lambda([F,[[A,[A,X]],E]])=0.
\notag
\end{eqnarray}
From Lemma \ref{aef} it follows that $[F,[A,[A,X]]]$ 
and $[E,[A,[A,X]]]$ are
isotropic vectors for $\mathrm B_\lambda$, and
Lemma \ref{radicalelements} implies that 
\begin{equation}
\label{twozeros}
\lambda([E,[F,[A,[A,X]]]])=0\textrm{\ \ \ and\ \ \ }\lambda([F,[[A,[A,X]],E]])=0.
\end{equation}
From (\ref{twozeros}) and (\ref{anothjaq})
we obtain
$$
\lambda([X,[A,[A,X]]])=-\lambda([[A,[A,X]],[E,F]])=0
$$
which completes the proof of the lemma.

\end{proof}

\begin{lemma}
\label{difficultlemma} 
We have $[\germ n_1,\germ n_1]\subseteq\germ r_\lambda$.
\end{lemma}
\begin{proof}
We prove the statement by induction on the dimension of 
$\germ n$. Since $\germ n$ is nilpotent, it has a 
subalgebra $\germ n'=\germ n'_0\oplus\germ n'_1$ 
of codimension one. If $\germ n'_1=\germ n_1$, then
the lemma follows by the induction hypothesis applied to 
$\germ n'$. 
If $\dim\germ n'_1=\dim\germ n_1-1$, then 
we have shown that  
we can find an element $A\in\germ n_1$
such that we have a direct sum decomposition such as (\ref{directsm})
for which (\ref{orthogg}) 
holds. One can check that
\begin{eqnarray}
\label{sumofspaces}
[[\germ n_1,\germ n_1],[\germ n_1,\germ n_1]]&=&[[\germ n'_1,\germ n'_1],[\germ n'_1,\germ n'_1]]
+[[A,\germ n_1],[\germ n_1,\germ n_1]].
\end{eqnarray}
To complete the proof of Lemma \ref{difficultlemma}, we need to show that 
$$
[[\germ n'_1,\germ n'_1],[\germ n'_1,\germ n'_1]]
\subseteq\mathrm{ker}\,\lambda
\textrm{\ \ \ and\ \ \ }
[[A,\germ n_1],[\germ n_1,\germ n_1]]
\subseteq\mathrm{ker}\,\lambda.
$$
Induction hypothesis applied to $\germ n'$
immediately implies that 
$$
[[\germ n'_1,\germ n'_1],[\germ n'_1,\germ n'_1]]\subseteq\ker\lambda.
$$
Next we show that 
$$
[[A,\germ n_1],[\germ n_1,\germ n_1]]
\subseteq
\mathrm{ker}\,\lambda.
$$ 
To this end we prove that for every $B,C,D\in\germ n_1$ we have 
\begin{equation}
\label{maingoal}
\lambda([[A,B],[C,D]])=0.
\end{equation}
By the Jacobi identity we have 
$$
\lambda([[C,D],[A,B]])+\lambda([A,[B,[C,D]]])-\lambda([B,[[C,D],A]])=0.
$$
But as $[\germ n,\germ n]\subseteq\germ n'$, 
we have
$[B,[C,D]]\in\germ n'_1$ and therefore 
$$
\lambda([A,[B,[C,D]]])=0.
$$
By Lemma \ref{radicalelements}, to prove (\ref{maingoal}) 
it suffices to show that 
$$
\mathrm B_\lambda([A,[C,D]],[A,[C,D]])=0.
$$
The latter statement follows from Lemma \ref{firststrangelemma}.

\end{proof}

\begin{lemma}
\label{existenceofevenpart}
There exists a polarizing
subalgebra $\germ m_0$ of $\germ n_0$ 
corresponding to $\lambda$ 
such that $\germ m_0\supseteq [\germ n_1,\germ n_1]$.
\end{lemma}
\begin{proof}
Since $[\germ n_1,\germ n_1]$
is an ideal of $\germ n_0$, we can find a sequence of ideals of $\germ n_0$ such as
$$
\germ n_0=\germ i^{(1)}\supset\germ i^{(2)}\supset\germ i^{(3)}\supset\cdots\supset\germ i^{(r-1)}\supset\germ i^{(r)}=\{0\}
$$ 
such that for every $1<j\leq r$ we have 
$\dim\germ i^{(j-1)}=\dim\germ i^{(j)}+1$
and moreover
for some $1\leq s\leq r$ we have 
$[\germ n_1,\germ n_1]=\germ i^{(s)}$.
For every $1\leq j\leq r$ let 
$$
\omega_\lambda^{(j)}:\germ i^{(j)}\times\germ i^{(j)}\to\mathbb R
$$ 
be the skew-symmetric bilinear form 
defined by $\omega_\lambda^{(j)}(X,Y)=\lambda([X,Y])$
and let
$\germ q^{(j)}$ be the radical of $\omega_\lambda^{(j)}$.

By a result of M. Vergne (see \cite[Theorem 1.3.5]{corgr}) the vector space
$$
\germ q^{(1)}+\cdots+\germ q^{(r)}
$$ 
is indeed 
a polarizing Lie subalgebra of $\germ n_0$ corresponding to 
$\lambda$. 
Lemma \ref{difficultlemma} implies that 
$\germ q^{(s)}\supseteq [\germ n_1,\germ n_1]$. 

\end{proof}

\subsection{Existence of polarizing systems}
\label{polarisingsystems}

Throughout this section 
$(N_0,\germ m)$ will be a nilpotent super Lie group.   
Let $(M_0,\germ m,\Phi,C_0,\germ c,\lambda)$ 
be a polarizing system in $(N_0,\germ n)$ and 
$(\sigma_\mu,\rho^\sigma_\mu,\mathcal K_\mu)$ be a 
representation of $(C_0,\germ c)$ which is consistent
with this polarizing system. 
Let $\mathrm B_\lambda$ denote the bilinear
form on $\germ n_1$ defined in Section 
\ref{suitable}.

Obviously, for every 
$X\in\germ n_1$ we have 
$$
\mathrm B_\lambda(X,X)=\lambda([X,X])=\mu\circ\Phi([X,X])=\mu([\Phi(X),\Phi(X)])\geq 0.
$$
Consequently, $\mathrm B_\lambda$ is nonnegative definite.

Conversely, let $\lambda\in\germ n_0^*$ be such that 
$\mathrm B_\lambda$ is nonnegative definite.
From 
Lemma \ref{existenceofevenpart} 
it follows that 
there exists a sub super Lie group 
$(M_0,\germ m)$ of $(N_0,\germ n)$ 
such that $\germ m_1=\germ n_1$ and 
$\germ m_0$ is a polarizing Lie subalgebra of $\germ n_0$ 
corresponding to $\lambda$. Let 
$$
\germ k_\lambda=\{\,X\in\germ m_0\ |\ \lambda(X)=0\,\}
$$ 
and 
set
$\germ j=\germ k_\lambda\oplus\germ r_\lambda$. 
\begin{lemma}
\label{thevecisideal}
The vector space $\germ j$ is an ideal in $\germ m$. 
\end{lemma}
\begin{proof}
Since $\lambda([\germ m_0,\germ m_0])=0$, we have 
$[\germ m_0,\germ m_0]\subseteq \germ k_\lambda$ and therefore 
$[\germ m_0,\germ k_\lambda]\subseteq \germ k_\lambda$. 

Next we prove that $[\germ k_\lambda,\germ m_1]\subseteq\germ r_\lambda$.
To this end, first note that by the Jacobi identity
for every 
$A\in\germ k_\lambda$, $B\in\germ r_\lambda$, and $C\in\germ n_1$
we have 
$$
-[[A,B],C]-[[B,C],A]+[[C,A],B]=0
$$
and therefore
$$
-\lambda([[A,B],C])-\lambda([[B,C],A])+\lambda([[C,A],B])=0.
$$
But 
$\lambda([[B,C],A])=0$
because $[B,C]\in[\germ n_1,\germ n_1]\subseteq \germ m_0$,
and 
$$
\lambda([[C,A],B])=0
$$
because $B\in\germ r_\lambda$. Consequently, 
for every $A\in \germ k_\lambda$ and $B\in\germ r_\lambda$ 
we have $\mathrm{ad}_AB\in\germ r_\lambda$. 
It follows that $\mathrm{ad}_A$ descends to a linear transformation  
$\overline{\mathrm{ad}}_A:\germ n_1/\germ r_\lambda\to\germ n_1/\germ r_\lambda$. 
The bilinear form $\mathrm B_\lambda$ induces a 
positive definite bilinear form 
$$
\overline{\mathrm B_\lambda}:\germ n_1/\germ r_\lambda\times\germ  
n_1/\germ r_\lambda\to\mathbb R.
$$
Next observe that for every $A\in\germ k_\lambda$ and 
every 
$V,W\in \germ n_1$
we have
$$
-\lambda([W,[A,V]])+\lambda([A,[V,W]])+\lambda([V,[W,A]])=0.
$$
Moreover, $\lambda([A,[V,W]])=0$ since 
$[A,[V,W]]\in [\germ m_0,\germ m_0]\subseteq\germ k_\lambda$.
Therefore for every $v,w\in \germ n_1/\germ r_\lambda$ we have 
$$
\overline{\mathrm B_\lambda}(\overline{\mathrm{ad}}_Av,w)=
-\overline{\mathrm B_\lambda}(v,\overline{\mathrm{ad}}_Aw).
$$
In other words, 
$\overline{\mathrm{ad}}_A$ is skew-symmetric.
Since $\overline{\mathrm{ad}}_A$ is also nilpotent, it follows
that $\overline{\mathrm{ad}}_A=0$. Therefore 
$[\germ k_\lambda,\germ n_1]\subseteq\germ r_\lambda$.

Next we prove that $[\germ m_0,\germ r_\lambda]\subseteq \germ r_\lambda$.
To this end, first note that by the Jacobi identity, 
for every 
$A\in\germ m_0$, $B\in\germ r_\lambda$, and $C\in\germ n_1$ 
we have
$$
-\lambda([[A,B],C])-\lambda([[B,C],A])+\lambda([[C,A],B])=0.
$$
But $\lambda([[B,C],A])=0$ 
because 
$[[B,C],A]\in[\germ m_0,\germ m_0]\subseteq\germ k_\lambda$, and 
$$
\lambda([[C,A],B])=0
$$
because $B\in\germ r_\lambda$. It follows that
$\lambda([[A,B],C])=0$, and consequently, as 
$C\in\germ n_1$ is arbitrary, we have $[A,B]\in\germ r_\lambda$.

Finally, 
the inclusion $[\germ m_1,\germ r_\lambda]\subseteq \germ k_\lambda$
follows from the definition of $\germ k_\lambda$.

\end{proof}

\begin{lemma} 
The quotient Lie superalgebra $\germ m/\germ j$
is reduced. 
\end{lemma}
\begin{proof}
It suffices to prove that for every $X\in\germ n_1$ such that $[X,X]\in\germ k_\lambda$, we have
$X\in\germ r_\lambda$. But this follows immediately 
from Lemma \ref{radicalelements}.

\end{proof}
\begin{proposition}
\label{uniquepolar}
Let $(M_0,\germ m)$ be a sub super Lie group of $(N_0,\germ n)$ 
such that $\germ m_0$ is a polarizing subalgebra of $\germ n_0$ 
corresponding
to $\lambda$. Then
there exists a polarizing system $(M_0,\germ m,\Phi,C_0,\germ c,\lambda)$
and a representation 
$(\sigma_\mu,\rho^{\sigma_\mu},\mathcal K_\mu)$ of $(C_0,\germ c)$ which is
consistent with this polarizing system. Moreover, up to unitary 
equivalence and parity change the representation
$(\sigma_\mu\circ\Phi,\rho^{\sigma_\mu\circ\Phi},\mathcal K_\mu)$
of $(M_0,\germ m)$ is unique.
\end{proposition}
\begin{proof}
By Lemma \ref{nonzeropolar} the quotient Lie superalgebra $\germ m/\germ j$ is either zero or has a one dimensional
even part. 
Moreover, $\germ m/\germ j$ is zero if and only if $\lambda=0$.
Since the case $\lambda=0$ can be easily dealt with, 
from now on we assume that $\lambda\neq 0$, and consequently 
$\germ m/\germ j$ is nonzero.

Since $\germ m/\germ j$ is reduced and nilpotent,
we have $\mathcal Z(\germ m/\germ j)=\germ m_0/\germ k_\lambda$
and hence $\dim\mathcal Z(\germ m/\germ j)=1$.
Therefore from Proposition \ref{kirillovlemma} it follows that 
$\germ m/\germ j$ is of Clifford type. (Note that  
one may have $\dim\germ m/\germ j=1$.)

Let $K_\lambda$
be a closed subgroup of $M_0$ with Lie algebra $\germ k_\lambda$
and  
$$
\overline\Phi:(M_0,\germ m)\to(M_0/K_\lambda,\germ m/\germ j)
$$ 
be the natural quotient map. Then
$(M_0,\germ m,\overline\Phi,M_0/K_\lambda,\germ m/\germ k_\lambda,\lambda)$ is a polarizing system. 
Moreover, up to unitary equivalence and parity change, there exists a unique irreducible unitary representation 
$(\sigma_{\overline \mu},\rho^{\sigma_{\overline\mu}},\mathcal K_{\overline\mu})$
of 
$(M_0/K_\lambda,\germ m/\germ j)$ which is consistent with this polarizing system. 

Next we prove the uniqueness claim of Proposition \ref{uniquepolar}.
Without loss of generality, we can assume $\lambda\neq 0$. 
Consider another polarizing system 
$$
(M_0,\germ m,\Phi,C_0,\germ c,\lambda)
$$ 
and a consistent irreducible
unitary representation $(\sigma_\mu,\rho^{\sigma_\mu},\mathcal K_\mu)$ of
$(C_0,\germ c)$.
Observe that 
\begin{equation}
\label{interesting}
\germ j\subseteq\mathrm{ker}\,\Phi.
\end{equation}
Indeed for every $X\in\germ k_\lambda$ we have
$\mu\circ\Phi(X)=\lambda(X)=0$ which implies that
$\Phi(X)=0$. Similarly, for every $X\in\germ r_\lambda$ 
we have $\mu\circ\Phi([X,X])=\lambda([X,X])=0$ which implies that
$[\Phi(X),\Phi(X)]=\Phi([X,X])=0$. But since $\germ c$ is reduced, 
it follows that $\Phi(X)=0$. This completes the proof of
(\ref{interesting}).

From (\ref{interesting}) it follows that 
there exists an epimorphism 
$$
\Psi:(M_0/K_\lambda,\germ m/\germ j)\to(C_0,\germ c)
$$
which satisfies $ \Psi\circ\overline\Phi=\Phi$.
However, any 
epimorphism between super Lie groups of Clifford type is 
indeed an isomorphism. From
Proposition \ref{propcl} it follows that 
$$
(\sigma_{\overline\mu}\circ\overline\Phi,\rho^{\sigma_{\overline\mu}\circ\Phi},\mathcal K_{\overline\mu})\eqsim
(\sigma_\mu\circ\Psi\circ\overline\Phi,\rho^{\sigma_\mu\circ\Psi\circ\overline\Phi},\mathcal K_\mu)\eqsim
(\sigma_\mu\circ\Phi,\rho^{\sigma_\mu\circ\Phi},\mathcal K_\mu)
$$
which completes the proof.

\end{proof}
Let $(M_0,\germ m,\Phi,C_0,\germ c,\lambda)$ be a polarizing system
with a consistent representation 
$(\sigma_\mu,\rho^{\sigma_\mu},\mathcal K_\mu)$. 
From now on, 
$(\pi,\rho^\pi,\mathcal H)$ will denote the induced unitary
representation 
\begin{equation}
\label{piinduced}
(\pi,\rho^\pi,\mathcal H)=\mathrm{Ind}_{(M_0,\germ m)}^{(N_0,\germ n)}
(\sigma_\mu\circ\Phi,\rho^{\sigma_\mu\circ\Phi},\mathcal K_\mu).
\end{equation}
Recall the definition of $\germ a[\germ n]$ from Section
\ref{thereduced}. Since $\germ m$
is an ideal of $\germ n$, we have
\begin{equation}
\label{mcontainsa}
\germ m\supseteq\germ a[\germ n]
\end{equation}
because
$\germ m$ contains all of the 
generators of $\germ a[\germ n]$.
\begin{lemma}
\label{ainker}
$\germ a[\germ n]_0\cap\mathcal Z(\germ n)
\subseteq\mathrm{ker}\,\lambda$.
\end{lemma}
\begin{proof}
The proof of this lemma is essentially given throughout the proof of 
Case I of Theorem \ref{lastmainthm}, 
and therefore here we only give a sketch of the proof.
Let $(\pi,\rho^\pi,\mathcal H)$ be as in (\ref{piinduced}),
and assume $X\in\germ a[\germ n]_0\cap\mathcal Z(\germ n)$ and 
$\lambda(X)\neq 0$. Using the definition of the induced representation,
it is not difficult to see that $\pi^\infty(X)\neq 0$,
which contradicts
Proposition \ref{nilrad}.

\end{proof}

\subsection{Relation between $(\pi,\rho^\pi,\mathcal H)$ and $\lambda$}
Let $\lambda\in\germ n_0^*$ such that $\mathrm B_\lambda$ 
is nonnegative definite. By Proposition \ref{uniquepolar}
there exists a 
polarizing system $(M_0,\germ m,\Phi,C_0,\germ c,\lambda)$,
and by Theorem \ref{lastmainthm} the representation obtained 
by induction from a consistent representation of the polarizing system
is irreducible. Our next task is to
show that if we choose different polarizing systems, we always obtain the 
same representation.
\begin{proposition}
\label{threefive}
Up to unitary equivalence and parity change,
the representation $(\pi,\rho^\pi,\mathcal H)$ is uniquely 
determined by $\lambda$. 
\end{proposition}
\begin{proof}
We prove the proposition by induction on $\dim \germ n$. The argument
is similar to the proof of Theorem \ref{lastmainthm}.

Let $(M_0,\germ m,\Phi,C_0,\germ c,\lambda)$ (respectively,  
$(M'_0,\germ m',\Phi',C'_0,\germ c',\lambda)$) be a polarizing system
with a consistent representation 
$(\sigma_\mu,\rho^{\sigma_\mu},\mathcal K_\mu)$ 
(respectively, $(\sigma_{\mu'},\rho^{\sigma_{\mu'}},\mathcal K_{\mu'})$). 
(Note that the two polarizing systems are associated to the same $\lambda$.)
Suppose that 
$(\pi,\rho^\pi,\mathcal H)$ 
(respectively, $(\pi',\rho^{\pi'},\mathcal H')$)
is the induced representation defined as in (\ref{piinduced}).
Our main goal is to prove that
\begin{equation}
\label{maingoallast}
(\pi,\rho^\pi,\mathcal H)\eqsim(\pi',\rho^{\pi'},\mathcal H').
\end{equation}
There are three cases to consider.

{\it Case I: $(N_0,\germ n)$ is not reduced.} As in
the proof of Case I in  
Theorem \ref{lastmainthm},
we can show that 
$$
\germ a[\germ n]\cap\mathcal Z(\germ n)\neq\{0\}.
$$
Moreover, 
$\germ a[\germ n]\cap\mathcal Z(\germ n)\subseteq\germ m\cap\germ m'$,
and using Lemma \ref{ainker} we can see that for every 
$W\in\germ a[\germ n]\cap\mathcal Z(\germ n)$ we have $\Phi(W)=0$ and 
$\Phi'(W)=0$. 

Set $\germ s=\germ a[\germ n]\cap\mathcal Z(\germ n)$
and consider the corresponding sub super Lie group $(S_0,\germ s)$
of $(N_0,\germ n)$. 
Since $\germ s$ is an ideal of $\germ n$, we have a 
quotient homomorphism
$$
\mathbf q:(N_0,\germ n)\to(N_0/S_0,\germ n/\germ s).
$$
The polarizing system $(M_0,\germ m,\Phi,C_0,\germ c,\lambda)$
corresponds via $\mathbf q$ to a polarizing
system
$$
(M_0/S_0,\germ m/\germ s,\Phi_\mathbf q,
C_0,\germ c,\lambda_\mathbf q)
$$
in $(N_0/S_0,\germ n/\germ s)$, 
where $\lambda_\mathbf q\in(\germ n/\germ s)^*$ 
satisfies $\lambda_\mathbf q\circ\mathbf q=\lambda$.
If we set
$$
(\pi_\mathbf q,\rho^{\pi_\mathbf q},\mathcal H_\mathbf q)=
\mathrm{Ind}_{(M_0/S_0,\germ m/\germ s)}^{(N_0/S_0,\germ n/\germ s)}
(\sigma_\mu\circ\Phi_\mathbf q,\rho^{\sigma_\mu\circ\Phi_\mathbf q},
\mathcal K_\mu)
$$ 
then $(\pi,\rho^\pi,\mathcal H)\simeq
(\pi_\mathbf q\circ\mathbf q,\rho^{\pi_\mathbf q\circ\mathbf q},
\mathcal H_\mathbf q)$.
From the other polarizing system and its consistent representation one 
can obtain another representation
$(\pi'_\mathbf q,\rho^{\pi'_\mathbf q},\mathcal H'_\mathbf q)$
of $(N_0/S_0,\germ n/\germ s)$
which is defined in a similar way. Since 
$\dim\germ n/\germ s<\dim\germ n$, induction hypothesis
implies that
$$
(\pi_\mathbf q,\rho^{\pi_\mathbf q},\mathcal H_\mathbf q)
\eqsim
(\pi'_\mathbf q,\rho^{\pi'_\mathbf q},\mathcal H'_\mathbf q)
$$
from which (\ref{maingoallast}) follows immediately.

{\it Case II: $(N_0,\germ n)$ is reduced and 
$\mathcal Z(\germ n)\cap\mathrm{ker}\,\lambda\neq\{0\}$.}
In this case $\mathcal Z(\germ n)\cap\mathrm{ker}\,\lambda$ is 
an ideal of $\germ n$ and 
$\mathcal Z(\germ n)\cap\mathrm{ker}\,\lambda\subseteq\germ m\cap\germ m'$.
Set $\germ s=\mathcal Z(\germ n)\cap\mathrm{ker}\,\lambda$
and let $(S_0,\germ s)$ be the corresponding sub super Lie group of
$(N_0,\germ n)$. As in Case I above, using the quotient map
$$
\mathbf q:(N_0,\germ n)\to(N_0/S_0,\germ n/\germ s)
$$ 
we can obtain new polarizing systems and consistent 
representations for
$$
(N_0/S_0,\germ n/\germ s).
$$ The rest of the argument is similar
to that of Case I above.

{\it Case III:
$(N_0,\germ n)$ is reduced and 
$\mathcal Z(\germ n)\cap\mathrm{ker}\,\lambda=\{0\}$.}
In this case the proof is very similar to that of 
Case III in Theorem \ref{lastmainthm}.
Without loss of generality we can assume that $\germ n$
is not of Clifford type.
From 
$\mathcal Z(\germ n)\cap\mathrm{ker}\,\lambda=\{0\}$ it 
follows that $\dim \mathcal Z(\germ n)=1$. 
Let $X,Y,Z,$ and $\germ n'$ 
be as in part (a) of Proposition \ref{kirillovlemma}.
As shown in the proof of Case III in 
Theorem \ref{lastmainthm}, 
we can choose $X,Y,Z$ suitably such that there exist 
polarizing systems 
\begin{equation}
\label{poolpol}
(\overline M_0,\overline{\germ m},\overline\Phi,
C_0,\germ c,\overline\lambda)
\textrm{\ \ and\ \ }
(\overline{M}'_0,\overline{\germ m}',\overline{\Phi}',
C'_0,\germ c',\overline\lambda)
\end{equation}
in $(N_0,\germ n)$
with the following properties.
\begin{itemize}
\item[(a)]  
$\overline\lambda=\mathrm{Ad}^*(n)(\lambda)$ 
for some $n\in N_0$.
\item[(b)] $\overline{\germ m}\subseteq\germ n'$
and $\overline{\germ m}'\subseteq\germ n'$.
\item[(c)] $(\sigma_\mu,\rho^{\sigma_\mu},\mathcal K_\mu)$ 
is consistent with
$(\overline M_0,\overline{\germ m},\overline\Phi,
C_0,\germ c,\overline\lambda)$.
\item[(d)] 
$(\sigma_{\mu'},\rho^{\sigma_{\mu'}},\mathcal K_{\mu'})$
is consistent with
$(\overline M'_0,\overline{\germ m}',\overline{\Phi}',
C'_0,\germ c',\overline\lambda)$.

\item[(e)]  If $(\overline\pi,\rho^{\overline\pi},\overline{\mathcal H})=
\mathrm{Ind}_{(\overline M_0,\overline{\germ m})}^{(N_0,\germ n)}
(\sigma_\mu\circ\overline\Phi,\rho^{\sigma_\mu\circ\overline\Phi},\mathcal K_\mu)$ 
then
$$
(\overline\pi,\rho^{\overline\pi},\overline{\mathcal H})\simeq
(\pi,\rho^{\pi},\mathcal H).
$$
\item[(f)] If 
$(\overline{\pi}',\rho^{\overline{\pi}'},\overline{\mathcal H}')=
\mathrm{Ind}_{(\overline{M}'_0,\overline{\germ m}')}^{(N_0,\germ n)}
(\sigma_{\mu'}\circ\overline{\Phi}',
\rho^{\sigma_{\mu'}\circ\overline{\Phi}'},\mathcal K_{\mu'})$ 
then 
$$
(\overline{\pi}',\rho^{\overline{\pi}'},\overline{\mathcal H}')\simeq
(\pi',\rho^{\pi'},\mathcal H').
$$
\end{itemize}
Let $(N'_0,\germ n')$ be the sub super Lie group
of $(N_0,\germ n)$ corresponding to $\germ n'$.
Since $\dim\germ n'<\dim\germ n$, 
by induction hypothesis we have
$$
\mathrm{Ind}_{(\overline M_0,\overline{\germ m})}^{(N'_0,\germ n')}
(\sigma_{\mu}\circ\overline\Phi,
\rho^{\sigma_\mu\circ\overline\Phi},\mathcal K_{\mu})
\eqsim
\mathrm{Ind}_{(\overline M'_0,\overline{\germ m}')}^{(N'_0,\germ n')}
(\sigma_{\mu'}\circ\overline\Phi',
\rho^{\sigma_{\mu'}\circ\overline\Phi'},\mathcal K_{\mu'})
$$
and (\ref{maingoallast})
follows by Proposition \ref{indtransitivity}.

\end{proof}
\subsection{Geometric parametrization of 
representations}
Let $(\pi,\rho^\pi,\mathcal H)$ be an irreducible unitary 
representation of a nilpotent super Lie group $(N_0,\germ n)$.
One can associate a coadjoint orbit $\EuScript{O}\subseteq\germ n_0^*$
to $(\pi,\rho^\pi,\mathcal H)$ as follows.
Let $(\pi,\mathcal H)$ denote the restriction 
of $(\pi,\rho^\pi,\mathcal H)$ to $N_0$.
From Corollary \ref{mycorollary}
it follows that $(\pi,\mathcal H)$ is a direct sum of finitely many copies
of an irreducible unitary representation $(\sigma,\mathcal K)$
of $N_0$. By classical Kirillov theory 
\cite{corgr}, the representation
$(\sigma,\mathcal K)$ is associated to a coadjoint orbit 
$\EuScript{O}\subseteq\germ n_0^*$. 
Theorem \ref{thmindpol} shows that
$(\pi,\rho^\pi,\mathcal H)$ is induced from a consistent representation
of a polarizing system $(M_0,\germ m,\Phi,C_0,\germ c,\lambda)$ 
where $\lambda\in\EuScript{O}$.

Our last theorem puts together the results in this paper
to obtain a geometric parametrization of irreducible
unitary representations of $(N_0,\germ n)$ by coadjoint orbits. 
Recall that
$$
\germ n_0^{+}=\{\,\lambda\in\germ n_0^*\ |\ \mathrm B_\lambda
\textrm{ is nonnegative definite\,}\}.
$$
\begin{theorem}
\label{mthm}
For a nilpotent super Lie group $(N_0,\germ n)$, 
the process of associating a coadjoint 
orbit $\EuScript{O}\subseteq\germ n_0^*$
to an irreducible unitary representation 
$(\pi,\rho^\pi,\mathcal H)$
of $(N_0,\germ n)$
yields a 
bijection between equivalence classes of 
irreducible unitary representations 
(up to unitary equivalence and parity change)
and $N_0$-orbits in 
$\germ n_0^+$.
\end{theorem}
\begin{proof}
By part (a) of Theorem \ref{thmindpol},
any irreducible unitary representation 
$(\pi,\rho^\pi,\mathcal H)$ of 
$(N_0,\germ n)$ is induced from a 
consistent representation of a polarizing
system 
$$
(M_0,\germ m,\Phi,C_0,\germ c,\lambda)
$$
and as shown in Section \ref{polarisingsystems},
it follows that $\lambda\in\germ n_0^+$. 

By part (b) of 
Theorem \ref{thmindpol},
if $(\pi,\rho^\pi,\mathcal H)$ 
is induced from a consistent representation of another 
polarizing system 
$$
(M'_0,\germ m',\Phi',C'_0,\germ c',\lambda')
$$
then 
$\lambda$ and $\lambda'$ are in the same $N_0$-orbit.
Moreover, once we fix a $\lambda\in \germ n_0^+$, 
by Proposition \ref{uniquepolar}
there always exists an associated polarizing system and a consistent representation,
and by Proposition \ref{threefive}, up to unitary equivalence and 
parity change all 
such polarizing systems yield the same irreducible unitary 
representation of
$(N_0,\germ n)$. 

\end{proof}

{\noindent\bf Remark.} One can actually prove that for every 
irreducible unitary representation
$(\pi,\rho^\pi,\mathcal H)$
of $(N_0,\germ n)$, the space 
$[\germ n_1,[\germ n_1,\germ n_1]]$ 
acts trivially, i.e., 
\begin{equation}
\label{nonenonenone}
\rho^\pi(X)=0\textrm{ for every } 
X\in[\germ n_1,[\germ n_1,\germ n_1]].
\end{equation}
Indeed if $\lambda\in\germ n_0^+$ then
from Lemma \ref{difficultlemma}, Lemma \ref{existenceofevenpart},
and Lemma \ref{thevecisideal} it follows that 
$[\germ n_1,[\germ n_1,\germ n_1]]\subseteq\germ r_\lambda$.
Consequently, when $(\pi,\rho^\pi,\mathcal H)$ is induced from a 
consistent representation of a polarizing system 
$(M_0,\germ m,\Phi,C_0,\germ c,\lambda)$, we have
$\Phi([\germ n_1,[\germ n_1,\germ n_1]])=0$. 
Statement (\ref{nonenonenone}) now follows 
from the realization of the induced representation
given in Section \ref{realization}
and the fact that 
$[\germ n_1,[\germ n_1,\germ n_1]]$
is $N_0$-invariant.

Another more direct 
way to prove (\ref{nonenonenone}) is to
use the method of proof of
Proposition \ref{threefive}.
Statement (\ref{nonenonenone}) can be used to
obtain slightly different proofs for the 
main results of this paper.
We thank the referee for suggesting this statement
and the second method of proof.

\end{document}